\documentclass{amsart}

\usepackage{esint}
\usepackage{amsfonts}
\usepackage{amssymb}

\newtheorem{theorem}{Theorem}[section]

\newtheorem{corollary}[theorem]{Corollary}

\newtheorem{lemma}[theorem]{Lemma}

\newtheorem{proposition}[theorem]{Proposition}

\numberwithin{equation}{section}

\begin{document}
\title{Weighted Poincar\'e inequality and the Poisson Equation}
\author{Ovidiu Munteanu}
\email{ovidiu.munteanu@uconn.edu}
\address{Department of Mathematics, University of Connecticut, Storrs, CT
06268, USA}
\author{Chiung-Jue Anna Sung}
\email{cjsung@math.nthu.edu.tw}
\address{Department of Mathematics, National Tsing Hua University, Hsin-Chu,
Taiwan}
\author{Jiaping Wang}
\email{jiaping@math.umn.edu}
\address{School of Mathematics, University of Minnesota, Minneapolis, MN
55455, USA}

\begin{abstract}
We develop Green's function estimate for manifolds satisfying a weighted
Poincar\'e inequality together with a compatible lower bound on the Ricci
curvature. The estimate is then applied to establish existence and sharp
estimates of the solution to the Poisson equation on such manifolds. As an
application, Liouville property for finite energy holomorphic functions is
proven on a class of complete K\"ahler manifolds. Consequently, such
K\"ahler manifolds must be connected at infinity.
\end{abstract}

\thanks{The first author was partially supported by NSF grant DMS-1506220.
The second author was partially supported by MOST. The third author was
partially supported by NSF grant DMS-1606820.}
\maketitle

\section{Introduction}

Recently, in \cite{MSW}, we have studied the existence and estimates of the
solution $u$ to the Poisson equation

\begin{equation*}
\Delta u=-\varphi
\end{equation*}
on a complete Riemannian manifold $(M^n, g),$ where $\varphi $ is a given
smooth function on $M.$ Among other things, we have obtained the following
result.

\begin{theorem}
\label{P1}Let $\left( M^{n},g\right) $ be a complete Riemannian manifold
with bottom spectrum $\lambda _{1}(\Delta )>0$ and Ricci curvature 
$\mathrm{Ric}\geq -\left( n-1\right) K$ for some constant $K.$ Let $\varphi $ be a
smooth function such that

\begin{equation*}
\left\vert \varphi \right\vert \left( x\right) \leq c\,\left(1+r(x)\right)^{-k}
\end{equation*}
for some $k>1,$ where $r(x)$ is the distance function from $x$ to a fixed
point $p\in M.$ Then the Poisson equation $\Delta u=-\varphi $ admits a
bounded solution $u$ on $M.$

If, in addition, the volume of the unit ball $B(x,1)$ satisfies 
$\mathrm{V}\left( x,1\right) \geq v_{0}>0$ for all $x\in M,$ then the solution $u$
decays and

\begin{equation*}
\left\vert u\right\vert \left( x\right) \leq C\,\left( 1+r(x)\right) ^{-k+1}.
\end{equation*}
\end{theorem}

Recall that the bottom spectrum $\lambda _{1}(\Delta )$ or the smallest
spectrum of the Laplacian can be characterized as the best constant of the
Poincar\'{e} inequality

\begin{equation*}
\lambda _{1}(\Delta )\,\int_{M}\phi ^{2}dx\leq \int_{M}|\nabla \phi |^{2}dx.
\end{equation*}
It is known that $\lambda _{1}\left( \Delta \right) >0$ implies that $M$ is
non-parabolic, that is, there exists a positive symmetric Green's function 
$G\left( x,y\right) $ for the Laplacian. The preceding theorem relies on the
following sharp estimate of the minimal positive Green's function.

\begin{theorem}
\label{G1}Let $\left( M^{n},g\right) $ be an $n$-dimensional complete
manifold with $\lambda _{1}\left( \Delta \right) >0$ and $\mathrm{Ric}\geq
-\left( n-1\right) K.$ Then for any $p,x\in M$ and $r>0$ we have

\begin{equation*}
\int_{B\left( p,r\right) }G\left( x,y\right) dy\leq C\,\left( 1+r\right)
\end{equation*}
for some constant $C$ depending only on $n,$ $K$ and $\lambda _{1}\left(
\Delta \right).$
\end{theorem}

In the current paper, we continue to address similar issues for complete
manifolds satisfying more generally a so-called weighted Poincar\'{e}
inequality. Recall that Riemannian manifold $\left( M,g\right) $ satisfies a
weighted Poincar\'{e} inequality if there exists a function $\rho \left(
x\right) >0$ such that

\begin{equation}
\int_{M}\rho \phi ^{2}\leq \int_{M}|\nabla \phi |^{2}  \label{WP}
\end{equation}
for any compactly supported function $\phi \in C_{0}^{\infty }(M).$

Other than being a natural generalization of $\lambda _{1}\left( \Delta
\right) >0,$ there are various motivations for considering weighted 
Poincar\'{e} inequality. First, it is well-known (see \cite{LW1}) that $M$ being
nonparabolic is equivalent to the validity of the weighted Poincar\'{e}
inequality for some $\rho.$ Secondly, according to a result of Cheng 
\cite{C}, when the Ricci curvature of manifold $M$ is asymptotically nonnegative
at infinity, its bottom spectrum $\lambda _{1}\left( \Delta \right) =0,$ and
one is forced to work with weighted Poincar\'{e} inequalities. Thirdly, by
considering weighted Poincar\'{e} inequality, it enables one to consider
manifolds with Ricci curvature bounded below by a function. Typically, in
geometric analysis, one assumes the curvature to be bounded by a constant so
that various comparison theorems become available. As demonstrated in 
\cite{LW1, LW4}, weighted Poincar\'{e} inequality allows one to go beyond this realm.
Indeed, they were able to prove some structure theorems for manifolds with
its Ricci curvature satisfying the inequality 
\begin{equation*}
\text{\textrm{Ric}}(x)\geq -C\,\rho (x)
\end{equation*}
for a suitable constant $C$ for all $x\in M.$ Finally, weighted Poincar\'{e}
inequality occurs naturally under various geometric settings. Indeed, a
result of Minerbe \cite{Mi} (see \cite{H} for further development) implies that complete manifold $M$ with
nonnegative Ricci curvature satisfies weighted Poincar\'{e} inequality 
with $\rho (x)=c\,r^{-2}(x),$ where $r(x)$ is the distance from $x$ to a fixed
point $p$ in $M,$ provided that the following reverse volume comparison
holds for some constant $C$ and $\nu >2$

\begin{equation*}
\frac{\mathrm{V}(B(p,t))}{\mathrm{V}(B(p,s))}\geq C\,\left( \frac{t}{s}\right) ^{\nu }
\end{equation*}
for all $0<s<t<\infty.$ Also, for minimal submanifold $M^{n}$ of the
Euclidean space $\mathbb{R}^{N},$ weighted Poincar\'{e} inequality is valid
on $M$ with $\rho (x)=\frac{(n-2)^{2}}{4}\bar{r}^{-2}(x),$ where $\bar{r}(x)$
denotes the extrinsic distance function from $x$ to a fixed point (see \cite{Ca, LW1}). 
On the other hand, for a stable minimal hypersurface in a manifold with nonnegative
Ricci curvature, by the second variation formula, weighted Poincar\'{e}
inequality holds for $\rho (x)$ being the length square of the second
fundamental form.

We also remark that the weighted Poincar\'e inequalities in various forms
have appeared in many important issues of analysis and mathematical physics.
Agmon \cite{A} has used it in his study of eigenfunctions for the
Schr\"odinger operators. In the interesting papers \cite{F-P1} and \cite{F-P2}, 
Fefferman and Phong have considered the more general weighted
Sobolev type inequalities for pseudodifferential operators. There are many interesting
results concerning sharp form of the weight $\rho.$ The classical Hardy inequality
for the Euclidean space $\mathbb{R}^n$ implies that  $\rho (x)=\frac{(n-2)^{2}}{4}\,r^{-2}(x)$
and it is optimal. In \cite{BGG}, it is shown that a sharp $\rho$ on the hyperbolic space $\mathbb{H}^n$ is given by 
$\rho (x)=\frac{(n-1)^2}{4}+\frac{(n-2)^{2}}{4}\,r^{-2}(x).$ We also refer to \cite{DFP} for
a more systematic approach to finding an optimal $\rho$ for more general second order
elliptic operators.

Throughout the paper, we will assume the weight $\rho(x)$ in addition
satisfies both (\ref{P}) and (\ref{O}), that is, the $\rho $-metric defined
by

\begin{equation}
ds_{\rho }^{2}=\rho \,ds^{2}  \label{P}
\end{equation}
is complete; and for some constants $A>0$ and $\delta >0,$

\begin{equation}
\sup_{B\left( x,\frac{\delta }{\sqrt{\rho \left( x\right) }}\right) }\rho
\leq A\inf_{B\left( x,\frac{\delta }{\sqrt{\rho \left( x\right) }}\right)
}\rho  \label{O}
\end{equation}
for all $x\in M.$ 

We point out that these two conditions
obviously hold true for a weight of the form $\rho(x)=c\,r^{\alpha}(x)$ with $ \alpha\geq -2.$  
The metric $ds_{\rho }^{2}$ was first used by Agmon \cite{A} to study
decay estimates for eigenfunctions. It was later employed to establish $L^2$ decay estimates
for the Green's function in \cite{LW1}. 

Our first result is an integral estimate for the minimal positive Green's
function $G\left( x,y\right) $ on $M.$ In the following, we denote geodesic balls with respect to
the background metric $ds^{2}$ by $B\left(x,r\right),$ and to the metric 
$ds^2_{\rho}$ by $B_{\rho}\left( x,r\right).$ 

\begin{theorem}
\label{G2}Let $\left( M^{n},g\right) $ be a complete manifold satisfying the
weighted Poincar\'{e} inequality (\ref{WP}) with weight $\rho$ having
properties (\ref{P}) and (\ref{O}). Assume that $\mathrm{Ric}\geq -K\,\rho$
on $M$ for some $K\geq 0.$ Then

\begin{equation*}
\int_{B_{\rho }\left( p,r\right) }\rho \left( y\right) G\left( x,y\right)
dy\leq C\left( r+1\right)
\end{equation*}
for all $p$ and $x$ in $M,$ and all $r>0,$ where $C$ depends only on $n,$ 
$K,$ $\delta$ and $A.$
\end{theorem}

As an application of Theorem \ref{G2}, we obtain the following solvability
result for the Poisson equation.

\begin{theorem}
\label{P2}Let $\left( M^{n},g\right) $ be a complete manifold satisfying the
weighted Poincar\'{e} inequality (\ref{WP}) with weight $\rho$ having
properties (\ref{P}) and (\ref{O}). Assume that $\mathrm{Ric}\geq -K\rho $
on $M$ for some $K\geq 0.$ Then for smooth function $\varphi$ such that

\begin{equation*}
\left\vert \varphi \right\vert \left( x\right) \leq c\,\left(
1+r_{\rho}(x)\right) ^{-k}
\end{equation*}
for some $k>1,$ where $r_{\rho }(x)$ is the $\rho $-distance function from 
$x$ to a fixed point $p\in M,$ the Poisson equation 
$\Delta u=-\rho\,\varphi$ admits a bounded solution $u$ on $M.$

If, in addition, there exists $v_{0}>0$ such that 
\begin{equation*}
\mathcal{V}_{\rho }\left( x,1\right) =\int_{B_{\rho }\left( x,1\right) }\rho
\left( y\right) dy\geq v_{0}
\end{equation*}
for all $x\in M,$ then the solution $u$ decays and

\begin{equation*}
\left\vert u\right\vert \left( x\right) \leq C\,\left( 1+r_{\rho}(x)\right)^{-k+1}.
\end{equation*}
\end{theorem}

Obviously, these results are faithful generalization of the ones from 
$\lambda _{1}\left( \Delta \right) >0.$ We also point out that Theorem \ref{G2}
is sharp as remarked after the proof of Theorem \ref{G3}. In passing, we mention that
recently Catino, Monticelli and Punzo \cite{CMP} have studied the solvability of the Poisson
equation by only assuming the essential spectrum of $M$ is positive. In view of this,
one may speculate that some of the preceding results generalize with 
weighted Poincar\'e inequality holds only for smooth functions $\phi$ with support
avoiding a fixed geodesic ball. 

Our proof of Theorem \ref{G2} follows in part of that of Theorem \ref{G1}.
As in the proof of Theorem \ref{G1}, we write

\begin{eqnarray}
\int_{B_{\rho }\left( p,r\right) }\rho (y)\,G\left( x,y\right) dy
&=&\int_{B_{\rho }\left( p,r\right) \backslash B_{\rho }\left( x,1\right)
}\rho (y)\,G\left( x,y\right) dy  \label{r1} \\
&&+\int_{B_{\rho }\left( p,r\right) \cap B_{\rho }\left( x,1\right) }\rho
(y)\,G\left( x,y\right) dy.  \notag
\end{eqnarray}

Following \cite{MSW}, the integral over $B\left( p,r\right) \backslash
B\left( x,1\right)$ is estimated by the integral

\begin{equation*}
\int_{L_{x}\left( \alpha ,\beta \right) }\rho(y)\,G\left( x,y\right) dy
\end{equation*}
over the sublevel sets

\begin{equation*}
L_{x}\left( \alpha ,\beta \right) :=\left\{ y\in M:\alpha <G\left(
x,y\right) <\beta \right\} ,
\end{equation*}%
where $\alpha $ and $\beta $ are the minimum and maximum value of the
Green's function $G\left( x,y\right) $ over $B_{\rho }\left( p,r\right)
\backslash B_{\rho }\left( x,1\right) ,$ respectively. Using the weighted
Poincar\'{e} inequality instead of $\lambda _{1}\left( \Delta \right) >0$
and arguing as in \cite{MSW}, one obtains

\begin{equation*}
\int_{L_{x}\left( \alpha ,\beta \right) }\rho(y)\,G\left( x,y\right) dy \leq
c\,\int_{L_{x}\left( \frac{1}{2}\alpha ,2\beta \right) }G^{-1}\left(
x,y\right) \left\vert \nabla G\right\vert ^{2}\left(x,y\right) dy.
\end{equation*}
Now the co-area formula together with the fact that $G\left( x,y\right) $ is
harmonic on $M\backslash B_{\rho}\left( x,1\right)$ yields that

\begin{equation}
\int_{B_{\rho }\left( p,r\right) \backslash B_{\rho }\left( x,1\right) }\rho
\left( y\right) G\left( x,y\right) dy\leq C\left( r+1\right).  \label{r3}
\end{equation}

For the integral over $B_{\rho }\left( p,r\right) \cap B_{\rho }\left(
x,1\right) $ in (\ref{r1}), however, a different approach from \cite{MSW} is
needed. In the case of of $\lambda _{1}\left( \Delta \right) >0,$ the proof
relies on the following double integral estimate for the minimal positive
Green's function.

\begin{equation*}
\int_{A}\int_{B}G(x,y)\,dy\,dx\leq \frac{e^{\sqrt{\lambda _{1}\left( \Delta
\right) }}}{\lambda _{1}\left( \Delta \right) }\,\sqrt{\mathrm{V}\left(
A\right) }\sqrt{\mathrm{V}\left( B\right) }\left( 1+\,r(A,B)\right) e^{-%
\sqrt{\lambda _{1}\left( \Delta \right) }\,r(A,B)}
\end{equation*}
for any bounded domains $A$ and $B$ of $M$, where $r\left( A,B\right) $
denotes the distance between $A$ and $B$, and $\mathrm{V}\left( A\right), 
\mathrm{V}\left( B\right) $ their volumes.

Unfortunately, it is unclear to us at this point how to formulate and derive
a similar estimate under the weighted Poincar\'{e} inequality. To overcome
this difficulty, we decompose $B_{\rho }\left( p,r\right) $ into a sequence
of annuli and employ a similar argument as (\ref{r3}) for each annulus.
However, instead of the weighted Poincar\'{e} inequality, we now use Poincar%
\'{e} inequality by appealing to a result of Li and Schoen \cite{LS} on the
estimate of the bottom spectrum of a geodesic ball in terms of the Ricci
curvature lower bound and its radius. This argument has the added benefit
that it completely avoids the involvement of the heat kernel and treats the
two integrals of (\ref{r1}) away and near the singularity of the Green's
function in a unified manner.

The Green's function estimate in Theorem \ref{G2} leads to the following
volume comparison estimate for geodesic $\rho $-balls. Define 
\begin{equation}
\mathcal{V}_{\rho }\left( x,r\right) =\int_{B_{\rho }\left( x,r\right) }\rho
\left( y\right) dy.  \label{Mu_Intro}
\end{equation}

\begin{theorem}
\label{V_Intro}Let $\left( M^{n},g\right) $ be a complete manifold
satisfying the weighted Poincar\'{e} inequality (\ref{WP}) with weight $\rho 
$ having properties (\ref{P}) and (\ref{O}). Assume that $\mathrm{Ric}\geq
-K\rho $ on $M$ for some $K\geq 0.$ Then there exist constants $c_1$ and $c_2$ 
depending only on $n,$ $K,$ $\delta$ and $A$ such
that for all $x\in M,$

\begin{equation*}
c_1\,e^{2R}\mathcal{V}_{\rho}\left( x,1\right)  
\leq \mathcal{V}_{\rho }\left( x,R\right) \leq e^{c_2\, R} \,\mathcal{V}_{\rho }\left( x,1\right)
\end{equation*}
for all $1<R<\infty.$
\end{theorem}

We point out that the lower bound of the form  
$\mathcal{V}_{\rho }\left( x,R\right) \geq c\,e^{2R}$
first appeared in \cite{LW1}, where the constant $c$ may depend on $x.$

As an application of the solvability of the Poisson equation, we prove the
following result concerning the connectivity at infinity.

\begin{theorem}
\label{Ends}Let $\left( M,g\right) $ be a complete K\"{a}hler manifold
satisfying (\ref{WP}) with weight $\rho$ having properties (\ref{P}), 
(\ref{O}) and $\rho \leq C.$ Assume that there exists $v_{0}>0$ so that for all 
$x\in M$

\begin{equation*}
\mathcal{V}_{\rho }\left( x,1\right) =\int_{B_{\rho }\left( x,1\right) }\rho
\left( y\right) dy\geq v_{0}>0
\end{equation*}
and that the Ricci curvature lower bound $\mathrm{Ric}\geq -\zeta \rho $
holds for some function $\zeta \left( x\right) >0$ converging to zero at
infinity. Then $M$ has only one end.
\end{theorem}

The novelty of the result is that the assumption on the Ricci curvature is
essentially imposed only at infinity, yet we are able to conclude that the
manifold is connected at infinity. This is of course not true in the
Riemannian setting. Indeed, the connected sum of copies of $\mathbb{R}^{n}$
for $n\geq 3$ has non-negative Ricci curvature outside a compact set and
satisfies a weighted Poincar\'{e} inequality of the form 
$\rho(x)=c\,r^{-2}(x).$ Obviously, it can have as many ends as one wishes.

We remark that our assumption is vacuous when $\rho =\lambda _{1}\left(
\Delta \right) $ is constant according to the aforementioned result of Cheng 
\cite{C}. However, in the case $\lambda _{1}\left( \Delta \right) >0,$ there
are various results concerning the number of ends for both Riemannian and 
K\"{a}hler manifolds. We refer to the papers \cite{LW, LW3, LW2, M} for
more information and further references. It should also be noted, although
not explicitly stated there, that the argument in \cite{LW1} already implies that 
$M$ necessarily has finitely many ends, without assuming $M$ is K\"{a}hler.

To prove Theorem \ref{Ends}, we first observe the assumption that

\begin{equation*}
\mathcal{V}_{\rho }\left( x,1\right) =\int_{B_{\rho }\left( x,1\right) }\rho
\left( y\right) dy\geq v_{0}>0
\end{equation*}
ensures all ends of $M$ must be nonparabolic. Therefore, by the result of Li
and Tam \cite{LT}, $M$ admits a nonconstant bounded harmonic function $u$
with finite energy if it is not connected at infinity. According to \cite{L}, 
such $u$ must be pluriharmonic as $M$ is K\"ahler. One may view $u$ as a
holomorphic map from $M$ into the hyperbolic disk. The proof is then
completed by establishing a Liouville type result for such maps. It is
well-known from Yau's Schwarz lemma \cite{Y} that such map $u$ must be
constant if the Ricci curvature of the domain manifold $M$ is nonnegative.
The result was generalized by Li and Yau \cite{LY1} to address the case that
the negative part of the Ricci curvature of $M$ is integrable. They
concluded that $u$ is necessarily a constant map if $M$ is in addition
nonparabolic. Our next result may be viewed as further development along
this line.

\begin{theorem}
\label{Vanishing_Intro}Let $\left( M,g\right) $ be a complete K\"{a}hler
manifold satisfying the assumptions of Theorem \ref{Ends}. Assume that 
$F:M\rightarrow N$ is a finite energy holomorphic map into a complex
Hermitian manifold $N$ of non-positive bisectional curvature. Then $F$ 
must be a constant map.
\end{theorem}

The paper is organized as follows. In
Section \ref{Intro}, after making some preliminary observations relating 
$\rho$-balls to the background metric balls, we translate Poincar\'e
inequality, Sobolev inequality and gradient estimate from the background
metric balls to the $\rho$-balls. With these preparations, we prove 
Theorem \ref{G2} in Section \ref{G}. Section \ref{PP} is devoted to the Poisson
equation and the proof of Theorem \ref{P2}. In Section \ref{CW}, we discuss
applications of the Poisson equation and prove the Liouville property for
finite energy holomorphic maps. Section \ref{Co} contains a new treatment of
Theorem \ref{G1}. Comparing to the original proof in \cite{MSW}, we believe
the new one is more streamlined. The proof relies on estimates of heat
kernel and avoids level set consideration. 
It remains to be seen if this new approach can be adapted to handle
Theorem \ref{G2} as well.

\section{Properties of the $\protect\rho $-distance \label{Intro}}

In this section, we make preparations for proving 
Theorem \ref{G2} by relating both the geometry and analysis of the
$\rho$-balls to the background metric balls. Consider the $\rho$-distance
function, defined to be

\begin{equation*}
r_{\rho }(x,y)=\inf_{\gamma }l_{\rho }(\gamma ),
\end{equation*}
the infimum of the length with respect to metric $ds_{\rho }^{2}$ of all smooth 
curves joining $x$ and $y.$  For a fixed point $x\in M,$ one checks readily
that $|\nabla r_{\rho }|^{2}(x,y)=\rho (y).$ When there is no confusion, the $\rho$-distance from 
$x$ to a fixed point $p$ is simply denoted by $r_{\rho }\left( x\right).$ More generally,
for any function $v\in C^{1}\left( M\right),$ denote by $\nabla _{\rho }v$ the gradient of $v$ 
with respect to $ds_{\rho }^{2}.$ Then its length with respect to $ds_{\rho }^{2}$ is given
by $\left\vert \nabla _{\rho}v\right\vert _{\rho }^{2}=\frac{1}{\rho }\left\vert \nabla v\right\vert^{2}.$

We denote geodesic balls with center $x$ and radius $r$ with respect to $ds^{2}$ 
by $B\left(x,r\right)$ and those with respect to $ds_{\rho }^{2}$ by $B_{\rho}\left( x,r\right).$ 
Our first result shows that $B\left( x,\frac{r}{\sqrt{\rho \left( x\right) }}\right) $ 
and $B_{\rho }\left( x,r\right) $ are comparable when $r\leq 1.$ Without loss of generality, 
we may assume the constants $A$ and $\delta $ specified in (\ref{O}) satisfy $A>16$ and $\delta <1.$
Throughout this section, we use $c$ and $C$ to denote constants depending only on
dimension $n,$ the constant $K$ from the Ricci curvature lower bound, and the constants $A$ and 
$\delta $ in (\ref{O}). Any other dependencies will be explicitly stated.

\begin{proposition}
\label{E}Let $M$ be a complete Riemannian manifold satisfying weighted Poincar\'e inequality (\ref{WP})
with weight $\rho$ having properties (\ref{P}) and (\ref{O}). Then there
exists $C>0$ depending only on $A$ and $\delta $ such that for any $x\in M,$ 
\begin{equation*}
\sup_{B_{\rho }\left( x,1\right) }\rho \leq C\inf_{B_{\rho }\left(
x,1\right) }\rho .
\end{equation*}%
Furthermore, there exist $c_{0}>0$ and $C_{0}>0$ depending only on $A$ and 
$\delta $ such that

\begin{equation*}
B\left( x,\frac{c_{0}}{\sqrt{\rho \left( x\right) }}r\right) \subset B_{\rho
}\left( x,r\right) \subset B\left( x,\frac{C_{0}}{\sqrt{\rho \left( x\right) 
}}r\right)
\end{equation*}
for all $x\in M$ and $0<r\leq 1.$
\end{proposition}

\begin{proof}
Let $x\in M$ and $0<r\leq 1.$ Let $\tau \left( t\right),$ $0\leq t\leq T,$
be a minimizing $\rho$-geodesic starting from $x.$
We claim that either

\begin{equation}
\tau \left( \left[ 0,T\right] \right) \subset B\left( x,\frac{\delta }{\sqrt{\rho \left( x\right) }}r\right) 
\text{ \ or } l_{\rho }\left(\tau \right) >\frac{\delta }{A}r.  \label{e1}
\end{equation}

Indeed, if $\tau$ is not entirely
contained in $B\left( x,\frac{\delta }{\sqrt{\rho \left( x\right) }}r\right) ,$ 
then there exists $0<t_{1}<T$ so that 
$\tau \left( t\right) \in B\left( x,\frac{\delta }{\sqrt{\rho \left( x\right) }}r\right)$ 
for all $0\leq t\leq t_{1}$ and 
$\tau \left( t_{1}\right) \in \partial B\left( x,\frac{\delta }{\sqrt{\rho \left( x\right) }}r\right).$ 
Let $\bar{\tau}$ be the restriction of $\tau $ to $\left[ 0,t_{1}\right].$ Then
 
\begin{eqnarray*}
l_{\rho }\left( \bar{\tau}\right) &=&\int_{\bar{\tau}}\left\vert \bar{\tau}%
^{\prime }\right\vert _{\rho }\left( t\right) dt \\
&=&\int_{\bar{\tau}}\sqrt{\rho \left( \bar{\tau}\left( t\right) \right) }%
\left\vert \bar{\tau}^{\prime }\right\vert \left( t\right) dt \\
&\geq &\frac{1}{\sqrt{A}}\sqrt{\rho \left( x\right) }\int_{\bar{\tau}%
}\left\vert \bar{\tau}^{\prime }\right\vert \left( t\right) dt \\
&=&\frac{1}{\sqrt{A}}\,\sqrt{\rho \left( x\right) }\,l\left( \bar{\tau}\right),
\end{eqnarray*}
where in the third line we have used (\ref{O}) and that $\bar{\tau}\left(
t\right) \in B\left( x,\frac{\delta }{\sqrt{\rho \left( x\right) }}r\right) $
for all $t\leq t_{1}$. Since $\tau \left( t_{1}\right) \in \partial B\left(
x,\frac{\delta }{\sqrt{\rho \left( x\right) }}r\right) $, we have $l\left( 
\bar{\tau}\right) \geq \frac{\delta }{\sqrt{\rho \left( x\right) }}r.$
Consequently,

\begin{equation*}
l_{\rho }\left( \bar{\tau}\right) \geq \frac{\delta }{A}\,r.
\end{equation*}
This proves (\ref{e1}).

We infer from the claim that $r\left( x,y\right) <\frac{\delta }{\sqrt{\rho\left( x\right) }}\,r$
when $r_{\rho }\left( x,y\right) <\frac{\delta }{A}\,r.$ In other words,

\begin{equation}
B_{\rho }\left( x,\frac{\delta }{A}r\right) \subset B\left( x,\frac{\delta }{%
\sqrt{\rho \left( x\right) }}r\right)   \label{e2}
\end{equation}%
for all $x\in M$ and all $0<r\leq 1$. By (\ref{O}), this implies 
\begin{equation}
\sup_{B_{\rho }\left( x,\frac{\delta }{A}\right) }\rho \leq A\inf_{B_{\rho
}\left( x,\frac{\delta }{A}\right) }\rho.  \label{e3}
\end{equation}

Now for $x,y\in M$ with $r_{\rho }\left( x,y\right) \leq 1,$ let 
$\tau $ be a minimizing $\rho $-geodesic from $x$ to $y.$ 
Applying (\ref{e3}) successively on each interval of $\rho$-length $\frac{\delta }{A}$ along $\tau,$ 
we conclude that

\begin{equation*}
\frac{1}{C}\rho \left( x\right) \leq \rho \left( y\right) \leq C\rho \left(x\right),
\end{equation*}%
where $C=A^{\frac{2A}{\delta }}.$ Therefore,

\begin{equation}
\sup_{B_{\rho }\left( x,1\right) }\rho \leq C\inf_{B_{\rho }\left(
x,1\right) }\rho   \label{e7}
\end{equation}
for all $x\in M.$ This proves the first part of the proposition.

Note that by (\ref{e2}), for any $z_{1},z_{2}\in M$ and $0<r\leq 1,$

\begin{equation}
r\left( z_{1},z_{2}\right) <\frac{\delta }{\sqrt{\rho \left( z_{1}\right) }}\,r
\text{ whenever }r_{\rho }\left( z_{1},z_{2}\right) <\frac{\delta }{A}\,r.  \label{e4}
\end{equation}
So for $x,y\in M$ with $r_{\rho }\left( x,y\right) \leq r,$ applying (\ref{e4}) successively on intervals of 
$\rho$-length $\frac{\delta }{A}\,r$ along a minimizing $\rho $-geodesic $\tau $ from 
$x$ to $y$ and using (\ref{e7}), one concludes that

\begin{equation*}
r\left(x, y\right) \leq \frac{C_0}{\sqrt{\rho \left(x \right) }}\,r
\end{equation*}
for some $C_{0}>0$ depending on $A$ and $\delta.$ Hence,

\begin{equation}
B_{\rho }\left( x,r\right) \subset B\left( x,\frac{C_{0}}{\sqrt{\rho \left(
x\right) }}r\right)  \label{e9}
\end{equation}
for all $x\in M$ and $r\leq 1.$

We now show that 

\begin{equation}
B\left( x,\frac{c_{0}}{\sqrt{\rho \left( x\right) }}r\right) \subset B_{\rho
}\left( x,r\right)  \label{e8}
\end{equation}%
for all $x\in M$ and $r\leq 1$ with $c_{0}=\frac{\delta }{A}.$ 

Indeed, for $y\in B\left( x,\frac{c_{0}}{\sqrt{\rho \left( x\right) }}r\right) $ and
$\gamma \left( t\right),$ $0\leq t\leq T<\frac{c_{0}}{\sqrt{\rho \left( x\right) }}r,$
a minimizing geodesic joining $x$ and $y,$ we have

\begin{eqnarray*}
l_{\rho }\left( \gamma \right) &=&\int_{\gamma }\left\vert \gamma ^{\prime
}\right\vert _{\rho }\left( t\right) dt \\
&=&\int_{\gamma }\sqrt{\rho \left( \gamma \left( t\right) \right) }%
\left\vert \gamma ^{\prime }\right\vert \left( t\right) dt \\
&\leq &\sqrt{A}\sqrt{\rho \left( x\right) }\int_{\gamma }\left\vert \gamma
^{\prime }\right\vert \left( t\right) dt \\
&=&\sqrt{A}\sqrt{\rho \left( x\right) }l\left( \gamma \right)\\
&\leq& c_{0}\sqrt{A}\,r\\
&<& r,
\end{eqnarray*}%
where in the third line we have used (\ref{O}) together with $\gamma \left( t\right)
\in B\left( x,\frac{\delta }{\sqrt{\rho \left( x\right) }}r\right) $ for all 
$0\leq t\leq T.$  This proves (\ref{e8}).

From (\ref{e8}) and (\ref{e9}) we conclude that 

\begin{equation*}
B\left( x,\frac{c_{0}}{\sqrt{\rho \left( x\right) }}r\right) \subset B_{\rho
}\left( x,r\right) \subset B\left( x,\frac{C_{0}}{\sqrt{\rho \left( x\right) 
}}r\right) 
\end{equation*}%
for all $x\in M$ and $r\leq 1.$ This proves the proposition.
\end{proof}

The previous result enables us to translate some properties on geodesic balls
of metric $ds^{2}$ to those of $ds_{\rho }^{2}.$ Denote by 
$\lambda _{1}\left( B_{\rho }\left( x,r\right) \right)$ the first Dirichlet
eigenvalue of $B_{\rho }\left( x,r\right)$ with respect to metric $ds^2.$ Then

\begin{equation*}
\lambda _{1}\left( B_{\rho }\left( x,r\right) \right) \int_{B_{\rho }\left(
x,r\right) }\phi ^{2}\leq \int_{B_{\rho }\left( x,r\right) }\left\vert
\nabla \phi \right\vert ^{2}
\end{equation*}
for any $\phi \in C_{0}^{\infty }\left( B_{\rho }\left( x,r\right) \right).$
Here and in the following, all integrals are with respect to the Riemannian measure
induced by the metric $ds^2.$
Similarly, we use $C_{S}\left( B_{\rho }\left( x,r\right) \right)$ to denote
the optimal constant for the following Dirichlet Sobolev inequality
on $B_{\rho }\left(x,r\right).$ 

\begin{equation*}
C_{S}\left( B_{\rho }\left( x,r\right) \right) \left( \fint_{B_{\rho
}\left( x,r\right) }\phi ^{\frac{2n}{n-2}}\right)^{\frac{n-2}{n}}
\leq \fint_{B_{\rho }\left(x,r\right) }\left\vert \nabla \phi \right\vert ^{2}
+\frac{\rho(x)}{r^2}\,\fint_{B_{\rho }\left( x,r\right) }\phi ^{2}
\end{equation*}
for $\phi \in C_{0}^{\infty }\left( B_{\rho }\left( x,r\right) \right),$
where $\fint_{B_{\rho }\left( x,r\right) }u$ is the average value of
function $u$ over the set $B_{\rho }\left( x,r\right),$ namely,

\begin{equation*}
\fint_{B_{\rho }\left( x,r\right) }u=\frac{1}{\mathrm{V}\left( B_{\rho
}\left( x,r\right) \right) }\int_{B_{\rho }\left( x,r\right) }u
\end{equation*}
with $\mathrm{V}\left( B_{\rho }\left( x,r\right) \right) $ being the
volume of $B_{\rho }\left( x,r\right) $ with respect to metric $ds^{2}.$ 
We refer to $C_{S}\left( B_{\rho }\left( x,r\right) \right) $ as the Dirichlet Sobolev
constant for $B_{\rho }\left( x,r\right).$

\begin{lemma}
\label{L}Let $\left( M^{n},g\right) $ be a complete manifold satisfying (\ref{WP}),
(\ref{P}) and (\ref{O}). Assume that $\mathrm{Ric}\geq -K\rho $
on $M$ for some $K\geq 0.$ Then for some $C>0,$ 

\begin{eqnarray*}
\lambda _{1}\left( B_{\rho }\left( x,r\right) \right) &\geq &\frac{1}{Cr^{2}}%
\rho \left( x\right) \text{ } \\
C_{S}\left( B_{\rho }\left( x,r\right) \right) &\geq &\frac{1}{Cr^{2}}\rho
\left( x\right) 
\end{eqnarray*}
for any $x\in M$ and $0<r\leq \frac{\delta }{2C_{0}}.$ Here $C_{0}$ is the
constant specified in Proposition \ref{E}.
\end{lemma}

\begin{proof}
According to Li-Schoen \cite{LS}, if $\mathrm{Ric}\geq -H$ on $B\left(
x,2R\right) $, then

\begin{equation}
\lambda _{1}\left( B\left( x,R\right) \right) \geq \frac{1}{R^{2}}%
e^{-C\left( 1+R\sqrt{H}\right) }  \label{LS}
\end{equation}%
with $C$ depending only on dimension. For $r\leq \frac{\delta }{2C_{0}}$ we
have 

\begin{equation*}
B\left( x,\frac{2C_{0}}{\sqrt{\rho \left( x\right) }}r\right) \subset
B\left( x,\frac{\delta }{\sqrt{\rho \left( x\right) }}\right).
\end{equation*}%
The Ricci curvature lower bound assumption together with (\ref{O}) implies that
 
\begin{equation}
\mathrm{Ric}\geq -c\rho \left( x\right) \text{ on }B\left( x,\frac{2C_{0}}
{\sqrt{\rho \left( x\right) }}r\right).  \label{x8}
\end{equation}
Using (\ref{LS}) and (\ref{x8}) we get 

\begin{equation*}
\lambda _{1}\left( \text{ }B\left( x,\frac{C_{0}}{\sqrt{\rho \left( x\right) 
}}r\right) \right) \geq \frac{1}{Cr^{2}}\rho \left( x\right) 
\end{equation*}
for any $x\in M$ and $0<r\leq \frac{\delta }{2C_{0}}.$ Since
Proposition \ref{E} asserts

\begin{equation}
B_{\rho }\left( x,r\right) \subset B\left( x,\frac{C_{0}}{\sqrt{\rho \left(
x\right) }}r\right),  \label{x8.1}
\end{equation}
it follows that
\begin{equation*}
\lambda _{1}\left( B_{\rho }\left( x,r\right) \right) \geq \frac{1}{Cr^{2}}\rho \left( x\right) 
\end{equation*}
for any $x\in M$ and $0<r\leq \frac{\delta }{2C_{0}}.$ This proves the
eigenvalue lower bound.

To prove the Sobolev constant bound, we use a result of Saloff-Coste \cite{SC}
that the following Sobolev inequality holds on $B\left( x,R\right)$ if
$\mathrm{Ric}\geq -H$ on $B\left( x,2R\right).$ 
 
\begin{eqnarray}
&&\frac{1}{R^{2}}e^{-C\left( 1+\sqrt{H}R\right) }\mathrm{V}\left( B\left(
x,R\right) \right) ^{\frac{2}{n}}\left( \int_{B\left( x,R\right) }\phi ^{%
\frac{2n}{n-2}}\right) ^{\frac{n-2}{n}}  \label{cs1} \\
&\leq &\int_{B\left( x,R\right) }\left\vert \nabla \phi \right\vert ^{2}+%
\frac{1}{R^{2}}\int_{B\left( x,R\right) }\phi ^{2}  \notag
\end{eqnarray}%
for any $\phi \in C_{0}^{\infty }\left( B\left( x,R\right) \right).$ 

Now for $R=\frac{C_{0}}{\sqrt{\rho \left( x\right) }}\,r,$ in view of (\ref{x8}), 
applying (\ref{cs1}), we get

\begin{equation*}
\frac{1}{C}\frac{\rho \left( x\right) }{r^{2}}\mathrm{V}\left( B\left( x,R\right) \right) ^{\frac{2}{n}}
\left( \int_{B\left(x,R\right) }\phi ^{\frac{2n}{n-2}}\right) ^{\frac{n-2}{n}}\leq 
\int_{B\left( x,R\right) }\left\vert \nabla \phi \right\vert ^{2}
+\frac{\rho\left( x\right) }{r^{2}}\int_{B\left( x,R\right) }\phi ^{2}
\end{equation*}%
for any $\phi \in C_{0}^{\infty }\left( B\left( x,R\right) \right).$

However, by (\ref{x8.1}), we have $B_{\rho }\left( x,r\right) \subset
B\left( x,R\right).$ It follows for $\phi \in C_{0}^{\infty }\left( B_{\rho }\left( x,r\right) \right)$
that 

\begin{equation*}
\frac{1}{C}\frac{\rho \left( x\right) }{r^{2}}\mathrm{V}\left( B_{\rho
}\left( x,r\right) \right) ^{\frac{2}{n}}\left( \int_{B_{\rho }\left(
x,r\right) }\phi ^{\frac{2n}{n-2}}\right) ^{\frac{n-2}{n}}\leq \int_{B_{\rho
}\left( x,r\right) }\left\vert \nabla \phi \right\vert ^{2}+\frac{\rho
\left( x\right) }{r^{2}}\int_{B_{\rho }\left( x,r\right) }\phi ^{2}.
\end{equation*}
This completes the proof of the lemma.
\end{proof}

A well known result of Cheng and Yau \cite{CY} says that for $u>0$
a harmonic function on $B\left( x,R\right),$
 
\begin{equation}
\sup_{B\left( x,\frac{R}{2}\right) }\left\vert \nabla \ln u\right\vert \leq
c\left( \sqrt{H}+\frac{1}{R}\right)  \label{CY0}
\end{equation}%
for some constant $c>0$ depending only on dimension $n$
provided that the Ricci curvature $\mathrm{Ric}\geq -H$ on $B\left( x,R\right)$ 
for some nonnegative constant $H.$ We now use Proposition \ref{E} to translate 
this estimate to $\rho$-balls. 

\begin{lemma}
\label{CY}Let $\left( M^{n},g\right) $ be a complete manifold satisfying (\ref{WP}),
(\ref{P}) and (\ref{O}). Assume that $\mathrm{Ric}\geq -K\rho $
on $M$ for some $K\geq 0.$ Then there exists $c>0$ such that for $u>0$ a harmonic
function on $B_{\rho }\left( x,r\right) $ with $0<r\leq 1,$ 
 
\begin{equation*}
\sup_{B_{\rho }\left( x,\frac{r}{2}\right) }\left\vert \nabla _{\rho }\ln
u\right\vert _{\rho }\leq \frac{c}{r}.
\end{equation*}%
Consequently, 
\begin{equation*}
u\left( y\right) \leq cu\left( z\right) 
\end{equation*}
for $y, z\in B_{\rho }\left( x,\frac{r}{2}\right).$
\end{lemma}

\begin{proof}
For $y\in B_{\rho }\left( x,\frac{r}{2}\right),$
the triangle inequality implies that $B_{\rho }\left( y,\frac{r}{2}\right) \subset
B_{\rho }\left( x,r\right).$ On the other hand, by Proposition \ref{E}, we have

\begin{equation*}
B\left( y,\frac{c_{0}}{2\sqrt{\rho \left( y\right) }}r\right) \subset
B_{\rho }\left( y,\frac{r}{2}\right).
\end{equation*}
Therefore, $u$ is harmonic on $B\left( y,\frac{c_{0}}{2\sqrt{\rho \left(y\right) }}r\right).$ 
Using (\ref{O}) one sees that $\mathrm{Ric}\geq -c\rho \left( y\right) $ on 
$B\left( y,\frac{c_{0}}{2\sqrt{\rho \left(y\right) }}r\right).$ In conclusion, by (\ref{CY0}), 

\begin{equation*}
\left\vert \nabla \ln u\right\vert \left( y\right) \leq \frac{c}{r}\sqrt{%
\rho \left( y\right) }.
\end{equation*}
This can be rewritten into
 
\begin{equation}
\left\vert \nabla _{\rho }\ln u\right\vert _{\rho }\left( y\right) \leq 
\frac{c}{r}.  \label{CY2}
\end{equation}

Integrating (\ref{CY2}) along a minimizing $\rho$-geodesic
joining $x$ and $y$ yields

\begin{equation*}
\frac{1}{c}u\left( y\right) \leq u\left( x\right) \leq cu\left( y\right) 
\end{equation*}
for $y\in B_{\rho }\left( x,\frac{r}{2}\right).$  
This obviously implies 

\begin{equation*}
u\left( y\right) \leq cu\left( z\right) 
\end{equation*}
for $y, z\in B_{\rho }\left( x,\frac{r}{2}\right).$
The lemma is proved.
\end{proof}

\section{Green's function estimates\label{G}}

With the preparations in the previous section, we now prove Theorem \ref{G2}.
Throughout this section, unless otherwise specified, we continue to use 
$c$ and $C$ to denote constants depending only on $n,$ $K,$ $\delta$ and $A.$ 

Let us first note the following simple consequence of Lemma \ref{CY} which will be
used repeatedly below. For any $0<r\leq 1$ and $y\in M\backslash B_{\rho
}\left( x,r\right),$ apply the local gradient estimate Lemma \ref{CY} to the harmonic
function $u\left( q\right) =G\left( x,q\right) $ on $B_{\rho }\left( y,r\right).$ 
Then

\begin{equation}
\sup_{z\in B_{\rho }\left( y,\frac{r}{2}\right) }\left\vert \nabla _{\rho
}\ln G\right\vert _{\rho }\left( x,z\right) \leq \frac{c}{r},  \label{x3}
\end{equation}%
where the gradient is computed with respect to variable $z.$
Consequently, we have 

\begin{equation}
G\left( x,z_{1}\right) \leq CG\left( x,z_{2}\right)   \label{x3.1}
\end{equation}
for all $z_{1}, z_{2}\in B_{\rho }\left( y,\frac{r}{2}\right).$

We first establish a local Harnack estimate. 

\begin{lemma}
\label{H}Let $\left( M^{n},g\right) $ be a complete manifold satisfying the
weighted Poincar\'{e} inequality (\ref{WP}) with weight $\rho $ having
properties (\ref{P}) and (\ref{O}). Assume that $\mathrm{Ric}\geq -K\rho $
on $M$ for some $K\geq 0.$ Then 

\begin{equation}
G\left( x, y\right) \leq C_{1}\left( \frac{r_{2}}{r_{1}}\right)
^{C_{1}}\,G\left( x, z\right)  \label{t2}
\end{equation}
for any $y\in \partial B_{\rho }\left( x,r\right) $ and $z\in \partial
B_{\rho }\left( x,s\right),$ where
\begin{equation*}
0< r_{1}\leq r, s\leq r_{2}\leq 1.
\end{equation*}
\end{lemma}

\begin{proof}
Suppose first that both $y, z\in \partial B_{\rho }\left( x,r\right)$ for
some $0<r\leq 1.$ Since the estimate (\ref{x3.1}) implies 
\begin{equation*}
G\left( x,y\right) \leq C\,G\left( x,z\right) 
\end{equation*}%
for $z\in B_{\rho }\left( y,\frac{1}{2}r\right),$ it suffices to
prove (\ref{t2}) for $y$ and $z$ satisfying 
\begin{equation*}
r_{\rho }\left( y,z\right) \geq \frac{1}{2}\,r.
\end{equation*}
Let $\tau \left( t\right) $ and $\eta \left( t\right),$ $0\leq t \leq r,$
be minimizing $\rho$-geodesics from $x$ to $y$ and $z,$ respectively.  
We claim that $r_{\rho }\left( y,\eta \right) \geq \frac{1}{4}r$
and $r_{\rho }\left( z,\tau \right) \geq \frac{1}{4}r.$ Indeed, suppose 
$r_{\rho }\left( y,\eta\left( t_{0}\right) \right) <\frac{1}{4}r$ for some
$t_{0}\in \left( 0,r\right).$ Since $r_{\rho }\left(
x,y\right) =r_{\rho }\left( x,z\right) =r$ and $r_{\rho }\left( y,z\right)
\geq \frac{1}{2}r,$ the triangle inequality implies 

\begin{eqnarray*}
r_{\rho }\left( z,\eta \left( t_{0}\right) \right) &\geq &r_{\rho }\left(
y,z\right) -r_{\rho }\left( y,\eta \left( t_{0}\right) \right) \\
&>&\frac{1}{4}r
\end{eqnarray*}%
and 
\begin{eqnarray*}
r_{\rho }\left( x,\eta \left( t_{0}\right) \right) &\geq &r_{\rho }\left(
x,y\right) -r_{\rho }\left( y,\eta \left( t_{0}\right) \right) \\
&>&\frac{3}{4}r.
\end{eqnarray*}%
Adding up these two inequalities we get 
\begin{eqnarray*}
r_{\rho }\left( x,z\right) &=&r_{\rho }\left( x,\eta \left( t_{0}\right)
\right) +r_{\rho }\left( \eta \left( t_{0}\right) ,z\right) \\
&>&r.
\end{eqnarray*}%
This contradiction shows that $r_{\rho }\left( y,\eta \right) \geq \frac{1}{4%
}r$ as claimed. The proof of $r_{\rho }\left( z,\tau \right) \geq \frac{1}{4}%
r$ is similar.

Consequently, $u\left( q\right) =G\left( y,q\right) $ is harmonic on $%
B_{\rho }\left( \eta \left( t\right) ,\frac{1}{4}r\right) $ for all $t\in %
\left[ 0,r\right] .$ It follows from (\ref{x3}) that 
\begin{equation}
G\left( y,x\right) \leq C\,G\left( y,z\right) .  \label{t3}
\end{equation}%
Similarly, as $r_{\rho }\left( z,\tau \right) \geq \frac{1}{4}r,$ the
function $u\left( q\right) =G\left( z,q\right) $ is harmonic on $B_{\rho
}\left( \tau \left( t\right) ,\frac{1}{4}r\right) $ for all $t\in \left[ 0,r%
\right].$ By (\ref{x3}) we get 
\begin{equation}
G\left( z,y\right) \leq C\,\,G\left( z,x\right).  \label{t4}
\end{equation}%
Combining (\ref{t3}) with (\ref{t4}) we conclude that 
\begin{equation}
G\left( x,y\right) \leq C\,\,G\left( x,z\right)  \label{t5}
\end{equation}%
as claimed in (\ref{t2}). This proves (\ref{t2}) when both $y,z\in \partial
B_{\rho }\left( x,r\right).$

Now let $y\in \partial B_{\rho }\left( x,r\right) $ and $z\in \partial
B_{\rho }\left( x,s\right) $ with 
\begin{equation*}
0<r_{1}\leq r,s\leq r_{2}\leq 1.
\end{equation*}%
Let us assume first that $r<s.$ Let $\eta \left( t\right),$ $0\leq t\leq s,$
be a minimizing $\rho$-geodesic from $x$ to $z.$
Applying (\ref{t5}) to $y\in \partial B_{\rho }\left( x,r\right) $ and $%
\eta \left( r\right) \in \partial B_{\rho }\left( x,r\right),$ we get that 
\begin{equation}
G\left( x,y\right) \leq C\,\,G\left( x,\eta \left( r\right) \right).
\label{t6}
\end{equation}%
Note that the function $u\left( q\right) =G\left( x,q\right) $ is harmonic on $B_{\rho
}\left( \eta \left( t\right) ,t\right) $ for all $r\leq t\leq s.$ Hence, according to (\ref{x3}), 

\begin{equation}
\left\vert \nabla _{\rho }\ln G\right\vert _{\rho }\left( x,\eta \left(
t\right) \right) \leq \frac{c}{t}.  \label{t6.1}
\end{equation}%
Integrating (\ref{t6.1}) in $t$ from $r$ to $s$ implies that 
\begin{equation*}
G\left( x,\eta \left( r\right) \right) \leq c\left( \frac{s}{r}\right)
^{c}G\left( x,z\right).
\end{equation*}%
Together with (\ref{t6}) and the fact $\frac{s}{r}\leq \frac{r_{2}}{r_{1}},$ one concludes

\begin{equation*}
G\left( x,y\right) \leq c\left( \frac{r_{2}}{r_{1}}\right) ^{c}G\left(
x,z\right) 
\end{equation*}%
for any $y\in \partial B_{\rho }\left( x,r\right) $ and $z\in \partial
B_{\rho }\left( x,s\right).$ This proves the result in the case $r\leq s.$

The remaining case of $s<r$ is similar, using (\ref{t6.1}) along
a minimizing $\rho$-geodesic $\tau \left( t\right) $ joining $x$ and $y$ instead.
\end{proof}

We now establish a similar result for any radius.

\begin{lemma}
\label{H1} Let $\left( M^{n},g\right) $ be a complete manifold satisfying
the weighted Poincar\'{e} inequality (\ref{WP}) with weight $\rho $ having
properties (\ref{P}) and (\ref{O}). Assume that $\mathrm{Ric}\geq -K\rho $
on $M$ for some $K\geq 0.$ Then 

\begin{equation*}
G\left( x,y\right) \leq e^{C\,r}G\left( x,z\right) 
\end{equation*}%
for any $p\in M,$ $x\in B_{\rho }\left( p,r\right),$ and any $y, z\in
B_{\rho }\left( p,r\right) \backslash B_{\rho }\left( x,1\right).$
\end{lemma}

\begin{proof}
For $y,z\in B_{\rho }\left( p,r\right) \backslash B_{\rho }\left( x,1\right),$ 
let $\tau \left( t\right),$ $0\leq t\leq T_1,$ and $\eta (\bar{t}),$ $0\leq \bar{t}\leq T_2,$
be minimizing $\rho$-geodesics from $x$ to $y$ and from $x$ to $z,$ respectively. Since $%
y,z\in B_{\rho }\left( p,r\right) $ and $r_{\rho }\left( p,x\right) <r,$ the
triangle inequality implies that $T_1, T_2<2r.$

Let $y_{1}=\tau \left( 1\right) \in \partial B_{\rho }\left(
x,1\right) $ and $z_{1}=\eta \left( 1\right) \in \partial B_{\rho }\left(
x,1\right) $ be the intersection points of $\tau $ and $\eta $ with $\partial
B_{\rho }\left( x,1\right).$ By Lemma \ref{H} we have

\begin{equation}
G\left( x,y_{1}\right) \leq c\,G\left( x,z_{1}\right).  \label{x5}
\end{equation}%
On the other hand, by (\ref{x3}),

\begin{equation}
\left\vert \nabla _{\rho }\ln G\right\vert _{\rho }\left( x,\tau \left(
t\right) \right) \leq c  \label{x6}
\end{equation}%
for all $1\leq t.$ Integrating (\ref{x6}) in $t$ from $1$ to $T_1$
yields that 

\begin{equation*}
G\left( x,y\right) \leq e^{cr}G\left( x,y_{1}\right).
\end{equation*}%
Similarly, we have 
\begin{equation*}
G\left( x,z_{1}\right) \leq e^{c\,r}\,G\left( x,z\right).
\end{equation*}%
In view of (\ref{x5}) we conclude that 
\begin{equation*}
G\left( x,y\right) \leq e^{cr}\,G\left( x,z\right).
\end{equation*}%
This proves the lemma.
\end{proof}

To prove Theorem \ref{G2}, we will need to consider the level sets of
the Green's function. Denote by
\begin{eqnarray*}
l_{x}\left( t\right) &=&\left\{ y\in M:G\left( x,y\right) =t\right\} \\
L_{x}\left( \alpha ,\beta \right) &=&\left\{ y\in M:\alpha <G\left(
x,y\right) <\beta \right\}.
\end{eqnarray*}

We will make extensive use of the following lemma. For
a proof,  see lemma 3.3 in \cite{MSW}.

\begin{lemma}
\label{C}Let $\left( M^{n},\,g\right) $ be a Riemannian manifold satisfying (%
\ref{WP}) with weight $\rho $ having property (\ref{P}). For any $t>0$ we
have 
\begin{equation*}
\int_{l_{x}\left( t\right) }\left\vert \nabla G\right\vert \left( x,\xi
\right) dA\left( \xi \right) =1,
\end{equation*}%
where $dA$ is the Riemannian area form of $l_{x}\left( t\right).$ 
Furthermore, for any $0<\alpha <\beta $ we have 
\begin{equation*}
\int_{L_{x}\left( \alpha ,\beta \right) }G^{-1}\left( x,y\right) \left\vert
\nabla G\right\vert ^{2}\left( x,y\right) dy=\ln \frac{\beta }{\alpha }.
\end{equation*}
\end{lemma}

A useful consequence of Lemma \ref{C} is that

\begin{equation}
\int_{L_{x}\left( \alpha ,\beta \right) }\rho \left( y\right) G\left(
x,y\right) dy\leq c\left( 1+\ln \frac{\beta }{\alpha }\right)   \label{t7}
\end{equation}
if the set $L_{x}\left( \frac{1}{e}\alpha, e\beta \right) $ is compact in $M.$
In fact, one only requires the weighted Poincar\'e inequality (\ref{WP})
to hold for smooth functions $\phi$ with support contained in 
$L_{x}\left( \frac{1}{e}\alpha, e\beta \right).$

Indeed, let $\phi $ be the cut-off function defined by 

\begin{equation*}
\phi \left( y\right) =\left\{ 
\begin{array}{c}
\ln \left( e\beta \right) -\ln G\,\left( x,y\right) \\ 
1 \\ 
\ln G\left( x,y\right) -\ln \left( \frac{1}{e}\alpha \right) \\ 
0%
\end{array}%
\right. 
\begin{array}{l}
\text{on }L_{x}\left( \beta ,e\beta \right) \\ 
\text{on }L_{x}\left( \alpha ,\beta \right) \\ 
\text{on }L_{x}\left( \frac{1}{e}\alpha ,\alpha \right) \\ 
\text{otherwise}%
\end{array}%
\end{equation*}%
Then the weighted Poincar\'e inequality (\ref{WP}) implies that

\begin{eqnarray}
\int_{M}\rho\left( y\right) \phi ^{2}\left( y\right) G\left(
x,y\right) dy &\leq &\int_{M}\left\vert \nabla \left( \phi G^{\frac{1}{2}%
}\right) \right\vert ^{2}\left( x,y\right) dy  \label{x9} \\
&\leq &\frac{1}{2}\int_{M}\phi ^{2}\left( y\right) \left\vert \nabla
G\right\vert ^{2}\left( x,y\right) G^{-1}\left( x,y\right) dy  \notag \\
&&+2\int_{M}G\left( x,y\right) \left\vert \nabla \phi \right\vert ^{2}\left(
y\right) dy  \notag
\end{eqnarray}%
Using the co-area formula and Lemma \ref{C}, we have 

\begin{eqnarray*}
&&\int_{M}\phi ^{2}\left( y\right) \left\vert \nabla G\right\vert ^{2}\left(
x,y\right) G^{-1}\left( x,y\right) dy \\
&\leq &\int_{L_{x}\left( \frac{1}{e}\alpha ,e\beta \right) }\left\vert
\nabla G\right\vert ^{2}\left( x,y\right) G^{-1}\left( x,y\right) dy \\
&=&2+\ln \left( \frac{\beta }{\alpha }\right).
\end{eqnarray*}%
The second term of the right hand side of (\ref{x9}) can be estimated as

\begin{eqnarray*}
&&\int_{M}G\left( x,y\right) \left\vert \nabla \phi \right\vert ^{2}\left(
y\right) dy \\
&\leq &\int_{L_{x}\left( \beta ,e\beta \right) }\left\vert \nabla
G\right\vert ^{2}\left( x,y\right) G^{-1}\left( x,y\right) dy \\
&&+\int_{L_{x}\left( \frac{1}{e}\alpha ,\alpha \right) }\left\vert \nabla
G\right\vert ^{2}\left( x,y\right) G^{-1}\left( x,y\right) dy \\
&=&2.
\end{eqnarray*}%
Combining these estimates we obtain 
\begin{equation*}
\int_{M} \rho\left( y\right) \phi ^{2}\left( y\right) G\left(
x,y\right) dy\leq c\left( 1+\ln \frac{\beta }{\alpha }\right) 
\end{equation*}%
as claimed in (\ref{t7}).

With a further cut-off, the assumption that the set $L_{x}\left( \frac{1}{e}\alpha ,e\beta \right) $ is
compact in $M$ is in fact not needed.

\begin{lemma}
\label{IG}Let $\left( M^{n},\,g\right) $ be a Riemannian manifold satisfying
(\ref{WP}) with weight $\rho $ having property (\ref{P}). Then
for all $0<\alpha <\beta ,$

\begin{equation*}
\int_{L_{x}\left( \alpha ,\beta \right) }\rho \left( y\right) G\left(
x,y\right) dy\leq c\left( 1+\ln \frac{\beta }{\alpha }\right).
\end{equation*}
\end{lemma}

\begin{proof}
The argument is similar to that of (\ref{t7}), the main difference being
that we use an additional cut-off in distance 
\begin{equation*}
\psi \left( y\right) =\left\{ 
\begin{array}{c}
1 \\ 
R+1-r_{\rho }\left( x,y\right) \\ 
0%
\end{array}%
\right. 
\begin{array}{l}
\text{on }B_{\rho }\left( x,R\right) \\ 
\text{on }B_{\rho }\left( x,R+1\right) \backslash B_{\rho }\left( x,R\right)
\\ 
\text{on }M\backslash B_{\rho }\left( x,R+1\right)%
\end{array}%
\end{equation*}%
Now define $\phi =\chi \psi $, where $\chi $ is given by 
\begin{equation*}
\chi \left( y\right) =\left\{ 
\begin{array}{c}
\ln \left( e\beta \right) -\ln G\,\left( x,y\right) \\ 
1 \\ 
\ln G\left( x,y\right) -\ln \left( \frac{1}{e}\alpha \right) \\ 
0%
\end{array}%
\right. 
\begin{array}{l}
\text{on }L_{x}\left( \beta ,e\beta \right) \\ 
\text{on }L_{x}\left( \alpha ,\beta \right) \\ 
\text{on }L_{x}\left( \frac{1}{e}\alpha ,\alpha \right) \\ 
\text{otherwise}%
\end{array}%
\end{equation*}%
We have 
\begin{eqnarray*}
\int_{M}\rho \left( y\right) \phi ^{2}\left( y\right) G\left( x,y\right) dy
&\leq &\int_{M}\left\vert \nabla \left( \phi G^{\frac{1}{2}}\right)
\right\vert ^{2}\left( x,y\right) dy \\
&\leq &\frac{1}{2}\int_{M}\left\vert \nabla G\right\vert ^{2}\left(
x,y\right) G^{-1}\left( x,y\right) dy \\
&&+4\int_{M}G\left( x,y\right) \left\vert \nabla \chi \right\vert ^{2}\left(
y\right) dy \\
&&+4\int_{M}G\left( x,y\right) \left\vert \nabla \psi \right\vert ^{2}\left(
y\right) \chi ^{2}\left( y\right) dy.
\end{eqnarray*}%
Using Lemma \ref{C}, as in the proof of (\ref{t7}), we get 
\begin{eqnarray}
\int_{M}\rho \left( y\right) \phi ^{2}\left( y\right) G\left( x,y\right) dy
&\leq &c\left( 1+\ln \frac{\beta }{\alpha }\right)  \label{x10} \\
&&+4\int_{M}G\left( x,y\right) \left\vert \nabla \psi \right\vert ^{2}\left(
y\right) \chi ^{2}\left( y\right) dy.  \notag
\end{eqnarray}%
Note that $\left\vert \nabla \psi \right\vert ^{2}\left( y\right) =\rho
\left( y\right) $ on its support. Since $G>e^{-1}\alpha $ on the support of $%
\chi ,$ we get that%
\begin{equation*}
\int_{M}G\left( x,y\right) \left\vert \nabla \psi \right\vert ^{2}\left(
y\right) \chi ^{2}\left( y\right) dy\leq \frac{e}{\alpha }\int_{B_{\rho
}\left( x,R+1\right) \backslash B_{\rho }\left( x,R\right) }\rho \left(
y\right) G^{2}\left( x,y\right) dy.
\end{equation*}%
However, it follows from Corollary 2.2 in \cite{LW1} (cf. Theorem 2.5 in 
\cite{MSW}) that 
\begin{eqnarray}
&&\int_{B_{\rho }\left( x,R+1\right) \backslash B_{\rho }\left( x,R\right)
}\rho \left( y\right) G^{2}\left( x,y\right) dy  \label{LW} \\
&\leq &Ce^{-2R}\int_{B_{\rho }\left( x,2\right) \backslash B_{\rho }\left(
x,1\right) }\rho \left( y\right) G^{2}\left( x,y\right) dy.  \notag
\end{eqnarray}%
In conclusion, this implies%
\begin{eqnarray}
&&\int_{M}G\left( x,y\right) \left\vert \nabla \psi \right\vert ^{2}\left(
y\right) \chi ^{2}\left( y\right) dy  \label{x11} \\
&\leq &\frac{C}{\alpha }e^{-2R}\int_{B_{\rho }\left( x,2\right) \backslash
B_{\rho }\left( x,1\right) }\rho \left( y\right) G^{2}\left( x,y\right) dy. 
\notag
\end{eqnarray}%
Combining (\ref{x10}) and (\ref{x11}) we obtain%
\begin{eqnarray*}
\int_{L_{x}\left( \alpha ,\beta \right) \cap B_{\rho }\left( x,R\right)
}\rho \left( y\right) G\left( x,y\right) dy &\leq &\int_{M}\rho \left(
y\right) \phi ^{2}\left( y\right) G\left( x,y\right) dy \\
&\leq &c\left( 1+\ln \frac{\beta }{\alpha }\right) \\
&&+\frac{C}{\alpha }e^{-2R}\int_{B_{\rho }\left( x,2\right) \backslash
B_{\rho }\left( x,1\right) }\rho \left( y\right) G^{2}\left( x,y\right) dy.
\end{eqnarray*}%
The result follows by taking $R\rightarrow \infty $ above.
\end{proof}

With the preceding lemmas, we now conclude the following.

\begin{proposition}
\label{MB}Let $\left( M^{n},g\right) $ be a complete manifold satisfying the
weighted Poincar\'{e} inequality (\ref{WP}) with weight $\rho $ having
properties (\ref{P}) and (\ref{O}). Assume that $\mathrm{Ric}\geq -K\rho $
on $M$ for some $K\geq 0.$ Then

\begin{equation*}
\int_{B_{\rho }\left( p,r\right) \backslash B_{\rho }\left( x,1\right) }\rho
\left( y\right) G\left( x,y\right) dy\leq C\left( r+1\right)
\end{equation*}
for any $p\in M$ and $x\in B_{\rho }\left( p,r\right).$
\end{proposition}

\begin{proof}
Let

\begin{equation*}
\alpha :=\inf_{y\in B_{\rho }\left( p,r\right) \backslash B_{\rho }\left(
x,1\right) }G\left( x,y\right) \text{ and }\beta :=\sup_{y\in B_{\rho
}\left( p,r\right) \backslash B_{\rho }\left( x,1\right) }G\left( x,y\right).
\end{equation*}
It follows from Lemma \ref{IG} that 

\begin{eqnarray*}
\int_{B_{\rho }\left( p,r\right) \backslash B_{\rho }\left( x,1\right)
}\rho(y)\,G\left( x,y\right) dy &\leq &\int_{L_{x}\left( \alpha ,\beta \right)
}\rho(y)\,G\left( x,y\right) dy \\
&\leq &c\left( \ln \frac{\beta }{\alpha }+1\right) .
\end{eqnarray*}%
However, Lemma \ref{H1} implies that 
\begin{equation*}
\beta \leq e^{c\,r}\alpha.
\end{equation*}%
The proposition follows.
\end{proof}

We now turn to the region around the pole and establish an integral estimate for the Green's function.

\begin{proposition}
\label{B}Let $\left( M^{n},g\right) $ be a complete manifold satisfying the
weighted Poincar\'{e} inequality (\ref{WP}) with weight $\rho $ having
properties (\ref{P}) and (\ref{O}). Assume that $\mathrm{Ric}\geq -K\rho $
on $M$ for some $K\geq 0.$ Then 

\begin{equation*}
\int_{B_{\rho }\left( x,1\right) }\rho \left( y\right) G\left( x,y\right)
dy\leq C
\end{equation*}
for all $x\in M.$
\end{proposition}

\begin{proof}
Let
\begin{equation}
\sigma \left( x\right) :=\inf_{y\in \partial B_{\rho }\left( x,1\right)
}G\left( x,y\right).  \label{S}
\end{equation}
Then by the maximum principle,

\begin{equation}
B_{\rho }\left( x,1\right) \subset L_{x}\left( \sigma \left( x\right)
,\infty \right).  \label{t8}
\end{equation}%
Hence, it suffices to prove that 
\begin{equation*}
\int_{L_{x}\left( \sigma \left( x\right) ,\infty \right) }\rho \left(
y\right) G\left( x,y\right) dy\leq C.
\end{equation*}%

First, observe that 

\begin{equation}
\sup_{y\in M\backslash B_{\rho }\left( x,r\right) }G\left( x,y\right)
=\sup_{y\in \partial B_{\rho }\left( x,r\right) }G\left( x,y\right).
\label{MG}
\end{equation}%
Indeed, being the minimal positive Green's function, $G(x,y)$ is the limit of $G_i(x,y),$
the Dirichlet Green's function of compact exhaustion $\Omega _{i}\subset M.$
Obviously,  

\begin{equation*}
\sup_{y\in \Omega_i\backslash B_{\rho }\left( x,r\right) }G_i\left( x,y\right)
=\sup_{y\in \partial B_{\rho }\left( x,r\right) }G_i\left( x,y\right).
\end{equation*}
After letting $i\rightarrow \infty,$ one sees that
(\ref{MG}) holds true for $G\left( x,y\right).$ In particular, 

\begin{equation}
\sup_{y\in \partial B_{\rho }\left( x,r\right) }G\left( x,y\right) \text{ is
decreasing in }r>0\text{.}  \label{in}
\end{equation}

By Lemma \ref{H}, there exists $C_{1}>0$ so that 
\begin{equation}
G\left( x,y\right) \leq C_{1}\left( \frac{r_{2}}{r_{1}}\right)
^{C_{1}}\,G\left( x,z\right)  \label{t10}
\end{equation}%
for any $y\in \partial B_{\rho }\left( x,r\right) $ and $z\in \partial
B_{\rho }\left( x,s\right)$ for
\begin{equation*}
0<r_{1}\leq r, s\leq r_{2}\leq 1.
\end{equation*}%
Hence, 
\begin{equation*}
\sup_{y\in \partial B_{\rho }\left( x,1\right) }G\left( x,y\right) \leq
C_{1}\sigma \left( x\right).
\end{equation*}%
For $C_{1}$ in (\ref{t10}), let
\begin{equation}
\omega :=C_{1}4^{C_{1}}.  \label{w}
\end{equation}%
Setting $s=1$ and $r=\frac{1}{2}$ in (\ref{t10}) we see that 

\begin{eqnarray*}
\sup_{y\in B_{\rho }\left( x,1\right) \backslash B_{\rho }\left( x,\frac{1}{2%
}\right) }G\left( x,y\right) &\leq &C_{1}2^{C_{1}}\sigma \left( x\right) \\
&<&\omega \sigma \left( x\right).
\end{eqnarray*}%
Together with (\ref{MG}), this proves that 

\begin{equation}
l_{x}\left( \omega \sigma \left( x\right) \right) \subset B_{\rho }\left( x,%
\frac{1}{2}\right).  \label{t11}
\end{equation}%

We now prove by induction that
 
\begin{equation}
l_{x}\left( \omega ^{k}\sigma \left( x\right) \right) \subset B_{\rho
}\left( x,\frac{1}{2^{k}}\right)   \label{t12}
\end{equation}
for all $k\geq 1.$

Assume (\ref{t12}) holds for some $k\geq 1.$ If it does not hold for $k+1,$ then
there exists 

\begin{equation}
y\in l_{x}\left( \omega ^{k+1}\sigma \left( x\right) \right) \cap \left(
B_{\rho }\left( x,\frac{1}{2^{k}}\right) \backslash B_{\rho }\left( x,\frac{1%
}{2^{k+1}}\right) \right),  \label{t12.1}
\end{equation}%
that is, $y\in \partial B_{\rho }\left( x,r\right) $ for 
$\frac{1}{2^{k+1}}<r\leq \frac{1}{2^{k}}$ and 
$G\left( x,y\right) =\omega^{k+1}\sigma \left( x\right).$
Now (\ref{t10}) and (\ref{w}) imply that

\begin{eqnarray*}
G\left( x,y\right) &\leq &C_{1}2^{C_{1}}\,G\left( x,z\right) \\
&<&\omega G\left( x,z\right)
\end{eqnarray*}
or
\begin{equation*}
G(x,z)> \omega ^{k}\sigma \left( x\right)
\end{equation*}
for all $z\in B_{\rho }\left( x,\frac{1}{2^{k}}\right) \backslash B_{\rho }\left( x,\frac{1}{2^{k+1}}\right).$
Therefore, by the maximum principle,

\begin{eqnarray*}
\min_{y\in B_{\rho}\left( x,\frac{1}{2^{k}}\right)}G\left( x,y\right)&=&
\min_{y\in B_{\rho }\left( x,\frac{1}{2^{k}}\right) \backslash B_{\rho }\left( x,\frac{1}{2^{k+1}}\right)}
G\left( x,y\right) \\
&>& \omega ^{k}\sigma \left( x\right).
\end{eqnarray*}
This violates the induction hypothesis that 
$l_{x}\left( \omega ^{k}\sigma \left( x\right) \right) \subset B_{\rho
}\left( x,\frac{1}{2^{k}}\right).$ 
So (\ref{t12}) is true for any $k\geq 1.$
In particular, we conclude that 

\begin{equation}
L_{x}\left( \frac{1}{e}\omega ^{k}\sigma \left( x\right) ,e\omega
^{k+1}\sigma \left( x\right) \right) \subset B_{\rho }\left( x,\frac{1}{%
2^{k-1}}\right)   \label{t13}
\end{equation}
and the set $L_{x}\left( \frac{1}{e}\omega ^{k}\sigma \left( x\right) ,e\omega ^{k+1}\sigma \left( x\right)
\right) $ is compact in $M$ for all $k\geq 2.$

Let 
\begin{equation}
k_{0}=\left[ \frac{\ln \left( C_{0}/\delta \right) }{\ln 2}\right] +3.
\label{k0}
\end{equation}
Then, for all $k\geq k_{0},$ $\frac{1}{2^{k-1}}\leq \frac{\delta }{2C_{0}}$ and
Lemma \ref{L} implies that 

\begin{equation*}
\lambda _{1}\left( B_{\rho }\left( x,\frac{1}{2^{k-1}}\right) \right) \geq 
\frac{1}{C}2^{2k}\rho \left( x\right).
\end{equation*}%
From this and (\ref{t13}) we infer that the Poincar\'{e} inequality 

\begin{equation*}
\frac{1}{C}2^{2k}\rho \left( x\right) \int_{M}\phi ^{2}\leq \int_{M}|\nabla
\phi |^{2}
\end{equation*}
holds for any compactly supported function $\phi \in C_{0}^{\infty
}(L_{x}\left( \frac{1}{e}\omega ^{k}\sigma \left( x\right) ,e\omega
^{k+1}\sigma \left( x\right) \right) )$ and any $k\geq k_{0}.$
Thus, applying (\ref{t7}), we get

\begin{equation*}
\rho \left( x\right) \int_{L_{x}\left( \omega ^{k}\sigma \left( x\right)
,\omega ^{k+1}\sigma \left( x\right) \right) }G\left( x,y\right) dy\leq 
\frac{C}{2^{2k}}\ln \omega.
\end{equation*}%
In view of (\ref{t13}) and Proposition \ref{E}, it may be written into

\begin{equation}
\int_{L_{x}\left( \omega ^{k}\sigma \left( x\right) ,\omega ^{k+1}\sigma
\left( x\right) \right) }\rho \left( y\right) G\left( x,y\right) dy\leq 
\frac{C}{2^{2k}}.  \label{t14}
\end{equation}%
Summing (\ref{t14}) over all $k\geq k_{0},$ we obtain 

\begin{eqnarray}
&&\int_{L_{x}\left( \omega ^{k_{0}}\sigma \left( x\right) ,\infty \right)
}\rho \left( y\right) G\left( x,y\right) dy  \label{t14.1} \\
&=&\sum_{k=k_{0}}^{\infty }\int_{L_{x}\left( \omega ^{k}\sigma \left(
x\right) ,\omega ^{k+1}\sigma \left( x\right) \right) }\rho \left( y\right)
G\left( x,y\right) dy  \notag \\
&\leq &C.  \notag
\end{eqnarray}

Note that Lemma \ref{IG} implies

\begin{equation}
\int_{L_{x}\left( \sigma \left( x\right) ,\omega ^{k_{0}}\sigma \left(
x\right) \right) }\rho \left( y\right) G\left( x,y\right) dy\leq C.
\label{t14.2}
\end{equation}%
Combining (\ref{t14.1}) and (\ref{t14.2}) we conclude that 

\begin{eqnarray*}
\int_{L_{x}\left( \sigma \left( x\right) ,\infty \right) }\rho \left(
y\right) G\left( x,y\right) dy &=&\int_{L_{x}\left( \sigma \left( x\right)
,\omega ^{k_{0}}\sigma \left( x\right) \right) }\rho \left( y\right) G\left(
x,y\right) dy \\
&&+\int_{L_{x}\left( \omega ^{k_{0}}\sigma \left( x\right) ,\infty \right)
}\rho \left( y\right) G\left( x,y\right) dy \\
&\leq &C.
\end{eqnarray*}
This completes the proof.
\end{proof}

We are now able to prove the main result of this section.

\begin{theorem}
\label{G3}Let $\left( M^{n},g\right) $ be a complete manifold satisfying the
weighted Poincar\'{e} inequality (\ref{WP}) with weight $\rho $ having
properties (\ref{P}) and (\ref{O}). Assume that $\mathrm{Ric}\geq -K\rho $
on $M$ for some $K\geq 0.$ Then, for any $p, x\in M,$ and $r>0,$ 

\begin{equation*}
\int_{B_{\rho }\left( p,r\right) }\rho \left( y\right) G\left( x,y\right)
dy\leq C\left( r+1\right).
\end{equation*}
\end{theorem}

\begin{proof}
We first remark that it suffices to prove the result for $x\in B_{\rho
}\left( p,r\right).$ Indeed, consider the function 

\begin{equation*}
\Phi \left( x\right) =\int_{B_{p}\left( p,r\right) }\rho \left( y\right)
G\left( x,y\right) dy.
\end{equation*}
We claim that the maximum value of $\Phi$ on $M\backslash B_{\rho }\left( p,r\right)$ 
must occur on $\partial B_{\rho }\left( p,r\right).$ This is because
$G\left( x,y\right) $ is the limit of $G_{i}\left( x,y\right),$ the
Dirichlet Green's function of compact exhaustion $\Omega _{i}$ of $M.$
If we let

\begin{equation*}
\Phi _{i}\left( x\right) =\int_{B_{p}\left( p,r\right) }\rho \left( y\right)
G_{i}\left( x,y\right) dy,
\end{equation*}%
then $\Phi _{i}\rightarrow \Phi $ as $i\rightarrow \infty.$ However,
by the maximum principle, 
the maximum value of $\Phi _{i}\left( x\right) $ on $\Omega_i\setminus B_{\rho }\left( p,r\right)$ 
is achieved on $\partial B_{\rho }\left( p,r\right).$ Therefore, 
the same is true for $\Phi \left( x\right).$

From now on, we assume that $x\in B_{\rho }\left( p,r\right).$
By Proposition \ref{MB} and Proposition \ref{B}, 
\begin{equation*}
\int_{B_{\rho }\left( p,r\right) \backslash B_{\rho }\left( x,1\right) }\rho
\left( y\right) G\left( x,y\right) dy\leq C\left( r+1\right)
\end{equation*}
and 
\begin{equation*}
\int_{B_{\rho }\left( x,1\right) }\rho \left( y\right) G\left( x,y\right)
dy\leq C.
\end{equation*}
Obviously, the theorem follows by combining these two estimates.
\end{proof}

Let us point out that Theorem \ref{G3} is sharp.  Indeed, for any 
$\varepsilon >0$ small enough so that $B\left( x,\varepsilon \right) \subset
B_{\rho }\left( x,t\right),$ we have

\begin{eqnarray*}
0 &=&\int_{B_{\rho }\left( x,t\right) \backslash B\left( x,\varepsilon
\right) }\Delta _{y}G\left( x,y\right) dy \\
&=&\int_{\partial B_{\rho }\left( x,t\right) }\frac{\partial G}{\partial \nu 
}\left( x,\xi \right) dA\left( \xi \right) \\
&&-\int_{\partial B\left( x,\varepsilon \right) }\frac{\partial G}{\partial r%
}\left( x,\xi \right) dA\left( \xi \right),
\end{eqnarray*}%
where $\nu $ is the unit normal of $\partial
B_{\rho }\left( x,t\right) $ with respect to $ds^{2}.$ 
Using the asymptotics of $G$ near its pole, we obtain

\begin{equation*}
\int_{\partial B\left( x,\varepsilon \right) }\frac{\partial G}{\partial r}%
\left( x,\xi \right) dA\left( \xi \right) =-1
\end{equation*}%
for any $\varepsilon >0.$ So

\begin{eqnarray}
1 &=&-\int_{\partial B_{\rho }\left( x,t\right) }\frac{\partial G}{\partial
\nu }\left( x,\xi \right) dA\left( \xi \right)  \label{GA} \\
&\leq &\int_{\partial B_{\rho }\left( x,t\right) }\left\vert \nabla
G\right\vert \left( x,\xi \right) dA\left( \xi \right)   \notag
\end{eqnarray}%
for any $t>0.$ Combining with the gradient estimate in (\ref{x3}) that 
\begin{equation*}
\left\vert \nabla G\right\vert \left( x,y\right) \leq C\sqrt{\rho \left(
y\right) }G\left( x,y\right)
\end{equation*}%
for $y\in M\backslash B_{\rho }\left( x,1\right),$ where the gradient
is taken in variable $y,$ we conclude

\begin{equation*}
\int_{\partial B_{\rho }\left( x,t\right) }\sqrt{\rho \left( \xi \right) }%
G\left( x,\xi \right) dA\left( \xi \right) \geq \frac{1}{C}
\end{equation*}%
for all $t\geq 1.$

Now the co-area formula yields

\begin{eqnarray}
&&\int_{B_{\rho }\left( x,r\right) \backslash B_{\rho }\left( x,1\right)
}\rho \left( y\right) G\left( x,y\right) dy  \label{GA1} \\
&=&\int_{1}^{r}\int_{\partial B_{\rho }\left( x,t\right) }\frac{1}{%
\left\vert \nabla r_{\rho }\right\vert \left( x,\xi \right) }\rho \left( \xi
\right) G\left( x,\xi \right) dA\left( \xi \right) dt  \notag \\
&=&\int_{1}^{r}\int_{\partial B_{\rho }\left( x,t\right) }\sqrt{\rho \left(
\xi \right) }G\left( x,\xi \right) dA\left( \xi \right) dt  \notag \\
&\geq &\frac{1}{C}\left( r-1\right).  \notag
\end{eqnarray}%
This shows that 
\begin{equation*}
\int_{B_{\rho }\left( x,r\right) }\rho \left( y\right) G\left( x,y\right)
dy\geq \frac{1}{C}\left( r-1\right) 
\end{equation*}%
for all $r>1,$ confirming the sharpness of Theorem \ref{G3}.

The above estimate of the Green's function leads to a volume comparison
result for geodesic $\rho $-balls. Define 

\begin{equation*}
\mathcal{V}_{\rho }\left( x,r\right) =\int_{B_{\rho }\left( x,r\right) }\rho
\left( y\right) dy.
\end{equation*}

\begin{theorem}
\label{V}Let $\left( M^{n},g\right) $ be a complete manifold satisfying the
weighted Poincar\'{e} inequality (\ref{WP}) with weight $\rho $ having
properties (\ref{P}) and (\ref{O}). Assume that $\mathrm{Ric}\geq -K\rho $
on $M$ for some $K\geq 0.$ Then for all $x\in M,$
 
\begin{equation*}
c\,e^{2R}\mathcal{V}_{\rho}\left( x,1\right) \leq \mathcal{V}_{\rho }\left( x,R\right) \leq \frac{e^{C\left( R+1\right) }}{r^{C}}
\mathcal{V}_{\rho }\left( x,r\right)
\end{equation*}
for all $0<r\leq 1\leq R.$
\end{theorem}

\begin{proof} 
We first prove the upper bound.
Theorem \ref{G3} implies that 

\begin{equation}
\int_{B_{\rho }\left( x,t\right) }\rho \left( y\right) G\left( x,y\right)
dy\leq C\left( t+1\right)   \label{t20}
\end{equation}%
for all $x\in M$ and $t>0.$ Set

\begin{eqnarray}
\sigma \left( x\right) &=&\inf_{y\in B_{\rho }\left( x,1\right) }G\left(
x,y\right)  \label{vd1} \\
&=&\inf_{y\in \partial B_{\rho }\left( x,1\right) }G\left( x,y\right).
\notag
\end{eqnarray}%
By Lemma \ref{H1},

\begin{equation*}
G\left( x,y\right) \geq e^{-cr_{\rho }\left( x,y\right) }\sigma \left(
x\right) 
\end{equation*}%
for $y\in M\backslash B_{\rho }\left( x,1\right).$ From (\ref{t20}) and
(\ref{vd1}) we conclude that 

\begin{eqnarray*}
C\left( t+1\right) &\geq &\int_{B_{\rho }\left( x,t\right) \backslash
B_{\rho }\left( x,t-1\right) }\rho \left( y\right) G\left( x,y\right) dy \\
&\geq &e^{-ct}\sigma \left( x\right) \int_{B_{\rho }\left( x,t\right)
\backslash B_{\rho }\left( x,t-1\right) }\rho \left( y\right) dy
\end{eqnarray*}%
for all $t\geq 1.$ Summing over $t$ from $1$ to $R,$ we get
 
\begin{equation}
\mathcal{V}_{\rho }\left( x,R\right) \leq \frac{e^{cR}}{\sigma \left(
x\right) } \label{vd3}
\end{equation}%
for all $x\in M$ and all $R\geq 1.$

On the other hand, according to (\ref{GA}), 

\begin{equation*}
\int_{\partial B_{\rho }\left( x,t\right) }\left\vert \nabla G\right\vert
\left( x,\xi \right) dA\left( \xi \right) \geq 1.
\end{equation*}%
In view of (\ref{x3}) we obtain for all $0<t\leq 1,$

\begin{equation}
1\leq \frac{C}{t}\int_{\partial B_{\rho }\left( x,t\right) }\sqrt{\rho
\left( \xi \right) }G\left( x,\xi \right) dA\left( \xi \right).  \label{vd4}
\end{equation}
Note Lemma \ref{H} implies for $0<t\leq 1,$

\begin{eqnarray*}
\sup_{y\in \partial B_{\rho }\left( x,t\right) }G\left( x,y\right) &\leq
&C\left( \frac{1}{t}\right) ^{C}\inf_{z\in \partial B_{\rho }\left(
x,1\right) }G\left( x,z\right) \\
&=&C\left( \frac{1}{t}\right) ^{C}\sigma \left( x\right).
\end{eqnarray*}%
Plugging into (\ref{vd4}) yields

\begin{equation}
\int_{\partial B_{\rho }\left( x,t\right) }\sqrt{\rho \left( \xi \right) }%
dA\left( \xi \right) \geq \frac{1}{C}\frac{t^{C}}{\sigma \left( x\right) }
\label{vd5}
\end{equation}%
for all $0<t\leq 1.$ 
So for any $0<r\leq 1,$ by the co-area formula,

\begin{eqnarray*}
&&\int_{B_{\rho }\left( x,r\right) \backslash B_{\rho }\left( x,\frac{r}{2}%
\right) }\rho \left( y\right) dy \\
&=&\int_{\frac{r}{2}}^{r}\int_{\partial B_{\rho }\left( x,t\right) }\frac{1}{%
\left\vert \nabla r_{\rho }\right\vert \left( x,\xi \right) }\rho \left( \xi
\right) dA\left( \xi \right) dt \\
&=&\int_{\frac{r}{2}}^{r}\int_{\partial B_{\rho }\left( x,t\right) }\sqrt{%
\rho \left( \xi \right) }dA\left( \xi \right) dt \\
&\geq &\frac{1}{C}r^{C}\frac{1}{\sigma \left( x\right) },
\end{eqnarray*}%
where in the last line we have used (\ref{vd5}). Thus, 
\begin{equation}
\frac{1}{\sigma \left( x\right) }\leq \frac{C}{r^{C}}\mathcal{V}_{\rho
}\left( x,r\right)   \label{vd6}
\end{equation}%
for all $r\leq 1.$

Combining (\ref{vd3}) and (\ref{vd6}) we conclude
\begin{equation*}
\mathcal{V}_{\rho }\left( x,R\right) \leq C\frac{e^{CR}}{r^{C}}\mathcal{V}%
_{\rho }\left( x,r\right) 
\end{equation*}%
for any $x\in M$ and $0<r\leq 1\leq R.$ This proves the upper bound.

We now turn to the lower bound. The same argument as in (\ref{GA1}) implies that 

\begin{equation*}
\frac{1}{C}\leq \int_{B_{\rho }\left( x,R\right) \backslash B_{\rho }\left(
x,R-1\right) }\rho \left( y\right) G\left( x,y\right) dy
\end{equation*}%
for $R>2.$ By the Cauchy-Schwarz inequality it
follows that 
\begin{equation*}
\frac{1}{C}\leq \mathcal{V}_{\rho }\left( x,R\right) \int_{B_{\rho }\left(
x,R\right) \backslash B_{\rho }\left( x,R-1\right) }\rho \left( y\right)
G^{2}\left( x,y\right) dy.
\end{equation*}%
Therefore, combining with (\ref{LW}), we obtain 
\begin{equation}
\frac{1}{C}e^{2R}\leq \mathcal{V}_{\rho }\left( x,R\right) \int_{B_{\rho
}\left( x,2\right) \backslash B_{\rho }\left( x,1\right) }\rho \left(
y\right) G^{2}\left( x,y\right) dy.  \label{lw2}
\end{equation}%
As in the proof of the upper bound, set 
\begin{equation*}
\sigma \left( x\right) =\inf_{y\in \partial B_{\rho }\left( x,1\right)
}G\left( x,y\right).
\end{equation*}%
Then Lemma \ref{H1} implies that

\begin{equation*}
\sup_{y\in B_{\rho }\left( x,2\right) \backslash B_{\rho }\left( x,1\right)
}G\left( x,y\right) \leq c\sigma \left( x\right).
\end{equation*}%
Hence, we obtain from (\ref{lw2}) that 
\begin{equation}
\frac{1}{C}e^{2R}\leq \sigma ^{2}\left( x\right) \mathcal{V}_{\rho }\left(
x,2\right) \mathcal{V}_{\rho }\left( x,R\right).  \label{lw3}
\end{equation}%
Applying (\ref{vd3}) for $R=1$ and using the upper bound we have 

\begin{equation}
\sigma ^{2}\left( x\right) \mathcal{V}_{\rho }\left( x,2\right) \leq \frac{C%
}{\mathcal{V}_{\rho }\left( x,1\right) }.  \label{lw4}
\end{equation}%
Clearly, (\ref{lw4}) and (\ref{lw3}) imply the lower bound.
\end{proof}

\section{\label{PP}The Poisson equation}

In this section, we focus on the Poisson equation and prove Theorem \ref{P2}.
We adopt the same convention that
$c$ and $C$ denote positive constants
depending on $n,$ $K,$ $\delta, $ and $A.$ We continue to
denote 
\begin{equation*}
r_{\rho }\left( x\right) =r_{\rho }\left( p,x\right).
\end{equation*}

\begin{theorem}
\label{P3}Let $\left( M^{n},g\right) $ be a complete manifold satisfying the
weighted Poincar\'{e} inequality (\ref{WP}) with weight $\rho $ having
properties (\ref{P}) and (\ref{O}). Assume that $\mathrm{Ric}\geq -K\rho $
on $M$ for some $K\geq 0.$ Then for any smooth function $\varphi $
satisfying 
\begin{equation*}
\left\vert \varphi \right\vert \left( x\right) \leq \omega \left( r_{\rho
}\left( x\right) \right),
\end{equation*}%
where $\omega \left( t\right) $ is a non-increasing function such that $%
\int_{0}^{\infty }\omega \left( t\right) dt<\infty,$ the Poisson equation $%
\Delta u=-\rho \varphi $ admits a bounded solution $u$ on $M$ with%
\begin{equation*}
\sup_{M}\left\vert u\right\vert \leq c\left( \omega \left( 0\right)
+\int_{0}^{\infty }\omega \left( t\right) dt\right).
\end{equation*}
\end{theorem}

\begin{proof}
We first prove that 
\begin{equation}
\int_{M}\rho \left( y\right) G\left( x,y\right) \left\vert \varphi
\right\vert \left( y\right) dy\leq c\left( \omega \left( 0\right)
+\int_{0}^{\infty }\omega \left( t\right) dt\right)  \label{s6}
\end{equation}%
for all $x\in M.$ Note that by Theorem \ref{G3} we have 
\begin{eqnarray*}
\int_{B_{\rho }\left( p,1\right) }\rho \left( y\right) G\left( x,y\right)
\left\vert \varphi \right\vert \left( y\right) dy &\leq &c\sup_{B_{\rho
}\left( p,1\right) }\left\vert \varphi \right\vert \\
&\leq &c\,\omega \left( 0\right) 
\end{eqnarray*}%
as $\omega $ is non-increasing. Therefore,

\begin{eqnarray}
&&\int_{M}\rho \left( y\right) G\left( x,y\right) \left\vert \varphi
\right\vert \left( y\right) dy  \label{s7} \\
&=&\sum_{j=0}^{\infty }\int_{B_{\rho }\left( p,2^{j+1}\right) \backslash
B_{\rho }\left( p,2^{j}\right) }\rho \left( y\right) G\left( x,y\right)
\left\vert \varphi \right\vert \left( y\right) dy  \notag \\
&&+\int_{B_{\rho }\left( p,1\right) }\rho \left( y\right) G\left( x,y\right)
\left\vert \varphi \right\vert \left( y\right) dy  \notag \\
&\leq &\sum_{j=0}^{\infty }\left( \int_{B_{\rho }\left( p,2^{j+1}\right)
\backslash B_{\rho }\left( p,2^{j}\right) }\rho \left( y\right) G\left(
x,y\right) dy\right) \sup_{B_{\rho }\left( p,2^{j+1}\right) \backslash
B_{\rho }\left( p,2^{j}\right) }\left\vert \varphi \right\vert  \notag \\
&&+c\omega \left( 0\right) .  \notag
\end{eqnarray}%
The hypothesis on $\varphi $ implies

\begin{equation*}
\sup_{B_{\rho }\left( p,2^{j+1}\right) \backslash B_{\rho }\left(
p,2^{j}\right) }\left\vert \varphi \right\vert \leq \omega \left(
2^{j}\right)
\end{equation*}%
and Theorem \ref{G3} says that 
\begin{equation*}
\int_{B_{\rho }\left( p,2^{j+1}\right) \backslash B_{\rho }\left(
p,2^{j}\right) }\rho \left( y\right) G\left( x,y\right) dy\leq c\,2^{j-1}.
\end{equation*}%
Using these estimates in (\ref{s7}) we obtain 
\begin{eqnarray*}
\int_{M}\rho \left( y\right) G\left( x,y\right) \left\vert \varphi
\right\vert \left( y\right) dy &\leq &c\omega \left( 0\right)
+c\,\sum_{j=0}^{\infty }2^{j-1}\,\omega \left( 2^{j}\right) \\
&\leq &c\omega \left( 0\right) +c\sum_{j=0}^{\infty
}\int_{2^{j-1}}^{2^{j}}\omega \left( t\right) dt \\
&\leq &c\left( \omega \left( 0\right) +\int_{0}^{\infty }\omega \left(
t\right) dt\right) .
\end{eqnarray*}%
This proves (\ref{s6}). As $\int_{0}^{\infty }\omega \left( t\right)
dt<\infty ,$ it follows that the function 
\begin{equation*}
u\left( x\right) :=\int_{M}\rho \left( y\right) G\left( x,y\right) \varphi
\left( y\right) dy
\end{equation*}%
is well defined, bounded on $M,$ and verifies 
\begin{equation*}
\Delta u=-\rho \varphi .
\end{equation*}%
Furthermore, we have the estimate 
\begin{equation*}
\sup_{M}\left\vert u\right\vert \leq c\left( \omega \left( 0\right)
+\int_{0}^{\infty }\omega \left( t\right) dt\right) .
\end{equation*}%
This proves the theorem.
\end{proof}

Our next step is to prove that the solution $u$ in Theorem \ref{P3} decays to
zero at infinity by assuming a uniform lower bound on $\mathcal{V}_{\rho
}\left( x,1\right),$ that is, 
\begin{equation}
\mathcal{V}_{\rho }\left( x,1\right) =\int_{B_{\rho }\left( x,1\right) }\rho
\left( y\right) dy\geq v_{0}>0  \label{Mu2}
\end{equation}%
for all $x\in M.$

We first establish a pointwise decay estimate for the Green's function. For the
rest of the section, constants $c$ and $C$ may in addition 
depend on $v_{0}.$

\begin{theorem}
\label{DG}Let $\left( M^{n},g\right) $ be a complete manifold satisfying the
weighted Poincar\'{e} inequality (\ref{WP}) with weight $\rho $ having
properties (\ref{P}), (\ref{O}), and (\ref{Mu2}). Assume that $\mathrm{Ric}%
\geq -K\rho $ on $M$ for some $K\geq 0.$ Then we have 
\begin{equation*}
G\left( x,z\right) \leq Ce^{-r_{\rho }\left( x,z\right) }
\end{equation*}%
for $z\in M$ with $r_{\rho }\left( x,z\right) \geq 1.$
\end{theorem}

\begin{proof}
By (\ref{LW}),

\begin{eqnarray}
&&\int_{B_{\rho }\left( x,r+1\right) \backslash B_{\rho }\left( x,r-1\right)
}\rho \left( y\right) G^{2}\left( x,y\right) dy  \label{t21} \\
&\leq &Ce^{-2r}\int_{B_{\rho }\left( x,3\right) \backslash B_{\rho }\left(
x,1\right) }\rho \left( y\right) G^{2}\left( x,y\right) dy \notag
\end{eqnarray}%
for any $r\geq 4.$ To estimate the right hand side of (\ref{t21}),
by Lemma \ref{H1} we have 

\begin{equation}
\sup_{y\in B_{\rho }\left( x,3\right) \backslash B_{\rho }\left( x,1\right)
}G\left( x,y\right) \leq c\inf_{y\in B_{\rho }\left( x,3\right) \backslash
B_{\rho }\left( x,1\right) }G\left( x,y\right).  \label{t22}
\end{equation}%
Together with Theorem \ref{G3}, it implies that 
\begin{eqnarray*}
C &\geq &\int_{B_{\rho }\left( x,3\right) \backslash B_{\rho }\left(
x,1\right) }\rho \left( y\right) G\left( x,y\right) dy \\
&\geq &\frac{1}{c}\left( \sup_{y\in B_{\rho }\left( x,3\right) \backslash
B_{\rho }\left( x,1\right) }G\left( x,y\right) \right) \int_{B_{\rho }\left(
x,3\right) \backslash B_{\rho }\left( x,1\right) }\rho \left( y\right) dy.
\end{eqnarray*}%
Consequently, 

\begin{equation}
\sup_{y\in B_{\rho }\left( x,3\right) \backslash B_{\rho }\left( x,1\right)
}G\left( x,y\right) \leq C\left( \int_{B_{\rho }\left( x,3\right) \backslash
B_{\rho }\left( x,1\right) }\rho \left( y\right) dy\right) ^{-1}.
\label{t23}
\end{equation}%
By (\ref{t23}) and (\ref{t21}) we get 
\begin{equation*}
\int_{B_{\rho }\left( x,r+1\right) \backslash B_{\rho }\left( x,r-1\right)
}\rho \left( y\right) G^{2}\left( x,y\right) dy\leq Ce^{-2r}\left(
\int_{B_{\rho }\left( x,3\right) \backslash B_{\rho }\left( x,1\right) }\rho
\left( y\right) dy\right) ^{-1}.
\end{equation*}%
But the hypothesis (\ref{Mu2}) implies
\begin{equation*}
\left( \int_{B_{\rho }\left( x,3\right) \backslash B_{\rho }\left(
x,1\right) }\rho \left( y\right) dy\right) ^{-1}\leq \frac{1}{v_{0}}.
\end{equation*}%
Therefore, we conclude

\begin{equation}
\int_{B_{\rho }\left( x,r+1\right) \backslash B_{\rho }\left( x,r-1\right)
}\rho \left( y\right) G^{2}\left( x,y\right) dy\leq Ce^{-2r} \label{t24}
\end{equation}%
for any $r\geq 4.$

For $z\in \partial B_{\rho }\left( x,r\right)$ with $r\geq 4,$ since 
\begin{equation*}
B_{\rho }\left( z,1\right) \subset B_{\rho }\left( x,r+1\right) \backslash
B_{\rho }\left( x,r-1\right),
\end{equation*}
it follows that

\begin{equation}
\int_{B_{\rho }\left( z,1\right) }\rho \left( y\right) G^{2}\left(
x,y\right) dy\leq Ce^{-2r_{\rho }\left( x,z\right) }.  \label{t25}
\end{equation}%
Using (\ref{x3}) that

\begin{equation*}
\left\vert \nabla _{\rho }G\left( x,y\right) \right\vert _{\rho }\leq c
\end{equation*}%
for all $y\in B_{\rho }\left( z,1\right),$ we have

\begin{equation*}
G\left( x,z\right) \leq c\inf_{y\in B_{\rho }\left( z,1\right) }G\left(
x,y\right).
\end{equation*}%
Plugging into (\ref{t25}), together with the hypothesis that 
\begin{equation*}
\mathcal{V}_{\rho }\left( z,1\right) \geq v_{0}>0,
\end{equation*}
one concludes

\begin{equation*}
G\left( x,z\right) \leq Ce^{-r_{\rho }\left( x,z\right) }
\end{equation*}%
for $z\in M$ with $r_{\rho }\left( x,z\right) \geq 4.$ This proves the
result.
\end{proof}

We now establish the decay estimate of the solution $u$ to the Poisson equation.

\begin{theorem}
\label{P4}Let $\left( M^{n},g\right) $ be a complete manifold satisfying the
weighted Poincar\'{e} inequality (\ref{WP}) with weight $\rho $ having
properties (\ref{P}), (\ref{O}), and (\ref{Mu2}). Assume that $\mathrm{Ric}%
\geq -K\rho $ on $M$ for some $K\geq 0.$ Then for any function $\varphi $
satisfying 
\begin{equation*}
\left\vert \varphi \right\vert \left( x\right) \leq \omega \left( r_{\rho
}\left( x\right) \right),
\end{equation*}%
where $\omega \left( t\right) $ is a non-increasing function such that $%
\int_{0}^{\infty }\omega \left( t\right) dt<\infty,$ the Poisson equation $%
\Delta u=-\rho \varphi $ admits a bounded solution $u$ on $M$ such that%
\begin{equation}
\left\vert u\right\vert \left( x\right) \leq C\left( \int_{\alpha r_{\rho
}\left( x\right) }^{\infty }\omega \left( t\right) dt\,+\,\mathcal{V}_{\rho
}\left( p,1\right) \omega \left( 0\right) e^{-\frac{1}{2}r_{\rho }\left(
x\right) }\right)   \label{d}
\end{equation}%
for all $x\in M,$ where $\alpha $ is a constant depending only on $n,$ $K,$ and $\delta,$ $A.$ 
\end{theorem}

\begin{proof}
According to Theorem \ref{V}, there exists a constant $c_{1}>0$ so that 
\begin{equation}
\mathcal{V}_{\rho }\left( p,t\right) \leq e^{c_{1}t}\mathcal{V}_{\rho
}\left( p,1\right)   \label{cv1}
\end{equation}%
for all $t\geq 1.$ For $c_{1}$ specified in (\ref{cv1}), set 
\begin{equation*}
\alpha =\frac{1}{2\left( c_{1}+1\right) }.
\end{equation*}

For $x\in M,$  let
\begin{equation}
R=r_{\rho }\left( x\right).  \label{R}
\end{equation}
We may assume $R\geq 6$ as the theorem obviously is true for $R\leq 6$ by adjusting the constant $C.$

Similar to Theorem \ref{P3} we have 
\begin{eqnarray*}
&&\int_{M\backslash B_{\rho }\left( p,\alpha R\right) }\rho \left( y\right)
G\left( x,y\right) \left\vert \varphi \right\vert \left( y\right) dy \\
&=&\sum_{j=0}^{\infty }\int_{B_{\rho }\left( p,2^{j+1}\alpha R\right)
\backslash B_{\rho }\left( p,2^{j}\alpha R\right) }\rho \left( y\right)
G\left( x,y\right) \left\vert \varphi \right\vert \left( y\right) dy \\
&\leq &\sum_{j=0}^{\infty }\left( \int_{B_{\rho }\left( p,2^{j+1}\alpha
R\right) \backslash B\left( p,2^{j}\alpha R\right) }\rho \left( y\right)
G\left( x,y\right) dy\right) \sup_{B_{\rho }\left( p,2^{j+1}\alpha R\right)
\backslash B_{\rho }\left( p,2^{j}\alpha R\right) }\left\vert \varphi
\right\vert \\
&\leq &C\sum_{j=0}^{\infty }\left( 2^{j-1}\alpha R\right) \,\omega \left(
2^{j}\alpha R\right),
\end{eqnarray*}%
where in the last line we have used the decay hypothesis on $\varphi $ and
Theorem \ref{G3}.

Since $\omega \left( t\right) $ is nonincreasing, it is easy to see that 
\begin{eqnarray*}
\sum_{j=0}^{\infty }\left( 2^{j-1}\alpha R\right) \omega \left( 2^{j}\alpha
R\right) &\leq &\sum_{j=0}^{\infty }\int_{2^{j-1}\alpha R}^{2^{j}\alpha
R}\omega \left( t\right) dt \\
&\leq &\int_{\frac{1}{2}\alpha R}^{\infty }\omega \left( t\right) dt.
\end{eqnarray*}%
It follows that%
\begin{equation}
\int_{M\backslash B_{\rho }\left( p,\alpha R\right) }\rho \left( y\right)
G\left( x,y\right) \left\vert \varphi \right\vert \left( y\right) dy\leq
c\int_{\frac{1}{2}\alpha R}^{\infty }\omega \left( t\right) dt.  \label{s9}
\end{equation}

We now proceed to obtain an estimate on $B_{\rho }\left( p,\alpha R\right).$
For $y\in B_{\rho }\left( p,j+1\right),$ where $0<j+1\leq R-2,$ we get by
triangle inequality that 
\begin{eqnarray*}
r_{\rho }\left( x,y\right) &\geq &r_{\rho }\left( p,x\right) -r_{\rho
}\left( p,y\right) \\
&\geq &R-\left( j+1\right).
\end{eqnarray*}%
Hence, by Theorem \ref{DG}, 
\begin{equation*}
G\left( x,y\right) \leq ce^{-\left( R-j\right) }
\end{equation*}%
for all $y\in B_{\rho }\left( p,j+1\right),$ where $0<j+1\leq R-2.$

Furthermore, by (\ref{cv1}),
\begin{equation*}
\mathcal{V}_{\rho }\left( p,j+1\right) \leq e^{c_{1}\left( j+1\right) }%
\mathcal{V}_{\rho }\left( p,1\right) 
\end{equation*}%
for any $j\geq 0.$ Combining these estimates together, we get 
\begin{equation}
\int_{B_{\rho }\left( p,j+1\right) \backslash B_{\rho }\left( p,j\right)
}\rho \left( y\right) G(x,y)dy\leq ce^{-\left( R-\left( c_{1}+1\right)
j\right) }\mathcal{V}_{\rho }\left( p,1\right)   \label{s10}
\end{equation}%
for all $0\leq j\leq R-3.$

Since $\alpha R\leq R-3,$ by (\ref{s10}) it follows that 
\begin{eqnarray*}
&&\int_{B_{\rho }\left( p,\alpha R\right) }\rho \left( y\right)
G(x,y)\,\left\vert \varphi \right\vert \left( y\right) dy \\
&\leq &\sum_{j=0}^{\left[ \alpha R\right] }\int_{B_{\rho }\left(
p,j+1\right) \backslash B_{\rho }\left( p,j\right) }\rho \left( y\right)
G(x,y)\,\left\vert \varphi \right\vert \left( y\right) dy \\
&\leq &c\,\mathcal{V}_{\rho }\left( p,1\right) \sum_{j=0}^{\left[ \alpha R%
\right] }e^{\left( c_{1}+1\right) j-R}\sup_{B_{\rho }\left( p,j+1\right)
\backslash B_{\rho }\left( p,j\right) }\left\vert \varphi \right\vert \\
&\leq &c\mathcal{V}_{\rho }\left( p,1\right) \omega \left( 0\right)
e^{-R\left( 1-\left( c_{1}+1\right) \alpha \right) } \\
&=&c\mathcal{V}_{\rho }\left( p,1\right) \omega \left( 0\right) e^{-\frac{1}{%
2}R},
\end{eqnarray*}%
where in the last line we have used that $\alpha =\frac{1}{2}\frac{1}{c_{1}+1}.$
Combining with (\ref{s9}) we arrive at

\begin{equation*}
\int_{M}\rho \left( y\right) G(x,y)\,\left\vert \varphi \right\vert \left(
y\right) dy\leq c\int_{\frac{1}{2}\alpha R}^{\infty }\omega \left( t\right)
dt+c\mathcal{V}_{\rho }\left( p,1\right) \omega \left( 0\right) e^{-\frac{1}{%
2}R}.
\end{equation*}%
This proves the theorem.
\end{proof}

Let us note that Theorem \ref{P2} follows from Theorems \ref{P3} and \ref{P4}. 
Indeed, in the case that the function $\varphi $ decays as 
\begin{equation*}
\left\vert \varphi \right\vert \left( x\right) \leq c\left( 1+r_{\rho
}\left( x\right) \right) ^{-k}
\end{equation*}%
for some $k>1$ and 
\begin{equation*}
\mathcal{V}_{\rho }\left( x,1\right) \geq v_{0}>0
\end{equation*}%
holds for all $x\in M,$ Theorem \ref{P4} readily implies that the solution $u$ satisfies 

\begin{equation*}
\left\vert u\right\vert \left( x\right) \leq C\left( k\right) \left(
1+r\left( x\right) \right) ^{-k+1}
\end{equation*}
as claimed in Theorem \ref{P2}.

\section{Applications \label{CW}}

In this section, we discuss some applications of the Poisson equation
and prove Theorem \ref{Ends}. We
continue to assume that $\left( M,g\right) $ is a complete manifold
satisfying the weighted Poincar\'{e} inequality (\ref{WP}), together with (\ref{P})
and (\ref{O}). Furthermore, we assume that there exists $v_{0}>0$ such
that the weighted volume 

\begin{equation}
\mathcal{V}_{\rho }\left( x,1\right) =\int_{B_{\rho }\left( x,1\right) }\rho
\left( y\right) dy\geq v_{0}>0  \label{Mu1}
\end{equation}%
for all $x\in M.$ In the following, unless otherwise specified, the constants $c$ and $C$
depend only on $n, K, \delta, A$ and $v_{0}.$ 

We begin with a Liouville type result.

\begin{theorem}
\label{Decay}Let $\left( M^{n},g\right) $ be a complete manifold satisfying
the weighted Poincar\'{e} inequality (\ref{WP}) with weight $\rho $ having
properties (\ref{P}), (\ref{O}), and (\ref{Mu1}), and $\mathrm{Ric}\geq -K\rho$ 
for some constant $K\geq 0.$ Let $\eta \geq 0$ be a $C^{1}$ function satisfying
\begin{equation*}
\eta \Delta \eta \geq -\zeta \rho \eta ^{2}+\left\vert \nabla \eta
\right\vert ^{2}
\end{equation*}%
for some positive continuous function $\zeta \left( x\right)$ which converges to
zero at infinity. If there exist $\varepsilon >0$ and $\Lambda >0$ such that 

\begin{equation}
\eta \left( x\right) \leq \Lambda e^{-\varepsilon r_{\rho }\left( x\right) } \label{hd}
\end{equation}
on $M,$ then $\eta =0$ on $M.$
\end{theorem}

\begin{proof}
We assume by contradiction that $\eta $ is not identically zero. We first
normalize $\eta $ by defining
 
\begin{equation}
h=\frac{1}{\Lambda e}\eta.  \label{k8}
\end{equation}
Then
\begin{equation*}
h\leq e^{-\varepsilon r_{\rho }-1}\ \text{\ on }M.
\end{equation*}%
As $h$ satisfies 
\begin{equation*}
\Delta h\geq -\zeta \rho h+\frac{\left\vert \nabla h\right\vert ^{2}}{h}
\end{equation*}%
at all points where $h>0,$ it is easy to see that 
\begin{equation}
\Delta \ln h\geq -\zeta \rho  \label{k7}
\end{equation}%
whenever $h>0.$ In addition, we have 
\begin{equation}
-\ln h\geq 1+\varepsilon r_{\rho }\text{ \ on }M.  \label{k9}
\end{equation}%
Denote by 
\begin{equation}
v=\frac{1}{\left( -\ln h\right) },  \label{k7.1}
\end{equation}%
where we set $v=0$ whenever $h=0.$ Hence, $v\in C^{0}\left( M\right).$

Computing directly, we have 
\begin{equation*}
\Delta v=\left( \Delta \ln h\right) v^{2}+2\left\vert \nabla \ln
h\right\vert ^{2}v^{3}.
\end{equation*}%
Hence, by (\ref{k7}) $v$ satisfies 
\begin{equation}
\Delta v\geq -\zeta \rho v^{2}  \label{k10}
\end{equation}%
whenever $v>0.$ Also, by (\ref{k9}),
 
\begin{equation}
0\leq v\leq \frac{1}{1+\varepsilon r_{\rho }}\text{ \ on }M.  \label{k11}
\end{equation}%
Define continuous function 
\begin{equation}
\varphi =\zeta v^{2}\text{ on }M  \label{k12}
\end{equation}%
and let
\begin{equation}
\omega \left( t\right) =\frac{1}{\left( 1+\varepsilon t\right) ^{2}}%
\sup_{M\backslash B_{\rho }\left( p,t\right) }\zeta.  \label{k13}
\end{equation}%
Clearly, $\omega $ is non-increasing and $\int_{0}^{\infty }\omega
\left( t\right) dt<\infty.$ Furthermore, (\ref{k11}) implies that 
\begin{equation*}
\left\vert \varphi \right\vert \left( x\right) \leq \omega \left( r_{\rho
}\left( x\right) \right) \text{ \ on }M.
\end{equation*}%
By Theorem \ref{P4}, the Poisson equation 
\begin{equation}
\Delta u=-\rho \varphi   \label{k14}
\end{equation}%
admits a bounded positive solution $u>0$ such that 
\begin{equation*}
0<u\left( x\right) \leq C\left( \int_{\alpha r_{\rho }\left( x\right)
}^{\infty }\omega \left( t\right) dt\,+\,\mathcal{V}_{\rho }\left(
p,1\right) \omega \left( 0\right) e^{-\frac{1}{2}r_{\rho }\left( x\right)
}\right) \text{ \ on }M
\end{equation*}%
for some $0<\alpha <1.$ Since $\phi $ is continuous, we have 
$u\in W_{\mathrm{loc}}^{2,p}\left( M\right) $ for any $p.$

By (\ref{k13}) we have that 
\begin{eqnarray*}
0 &<&u\left( x\right) \leq \frac{C}{1+\alpha \varepsilon r_{\rho }\left(
x\right) }\sup_{M\backslash B_{\rho }\left( p,\alpha r_{\rho }\left(
x\right) \right) }\zeta  \\
&&+\,C\mathcal{V}_{\rho }\left( p,1\right) e^{-\frac{1}{2}r_{\rho }\left(
x\right) }\sup_{M}\zeta.
\end{eqnarray*}%
As $\zeta \rightarrow 0$ at infinity we conclude that 
for any $\sigma >0$ there exists $R_{0}>0$ such that

\begin{equation}
u\left( x\right) \leq \frac{1}{\sigma r_{\rho }\left( x\right) }
\label{k15}
\end{equation}%
for all $x\in M\backslash B_{\rho }\left( p,R_{0}\right).$

We claim that 
\begin{equation}
v\leq u\text{ \ on }M.  \label{k16}
\end{equation}%
Suppose by contradiction that (\ref{k16}) is not true. Since by (\ref{k11})
and (\ref{k15}) both $u$ and $v$ approach $0$ at infinity, 
the function $v-u$ must achieve its maximum at some point 
$x_{0}\in M,$ where in particular $v\left( x_{0}\right) >0.$ Observe that by 
(\ref{k12}) and (\ref{k14}) we have $\Delta u=-\zeta \rho v^{2}$, whereas 
by (\ref{k10}) we have $\Delta v\geq -\zeta \rho v^{2}$ at any point where $v>0.$
Then $v-u\in W_{\mathrm{loc}}^{1,2}\left( M\right) $ is subharmonic in a
neighborhood of $x_{0}$ and achieves its maximum at $x_{0}.$ The strong maximum
principle implies that $v-u$ is in fact constant on $M.$ Obviously,
the constant must be $0.$ This contradiction implies that (\ref{k16}) is true.

In view of (\ref{k15}) and (\ref{k16}) we have proved that for any large $\sigma >0,$
there exists $R_{0}>0$ sufficiently large such that 
\begin{equation}
v\left( x\right) \leq \frac{1}{\sigma r_{\rho }\left( x\right) }\text{ \ for
all }x\in M\backslash B_{\rho }\left( p,R_{0}\right).  \label{k17}
\end{equation}

We now follow the proof of Theorem 4.4 in 
\cite{MSW} and show that $v$ decays faster than any polynomial order 
in the $\rho $-distance. This will be done by iterating the previous argument.

First, let us note the following fact. Define 
\begin{equation*}
\left\vert \zeta \right\vert _{\infty }:=\sup_{M}\zeta.
\end{equation*}%
Then (\ref{k10}) implies that 
\begin{equation}
\Delta v\geq -\left\vert \zeta \right\vert _{\infty }\rho v^{2}  \label{k18}
\end{equation}%
whenever $v>0.$ Assume that 
\begin{equation*}
v\left( x\right) \leq \theta \left( r_{\rho }\left( x\right) \right) 
\end{equation*}%
for some decreasing function $\theta \left( t\right) $ such that $%
\int_{0}^{\infty }\theta ^{2}\left( t\right) dt<\infty.$ Then there exists 
$0<\alpha <1$ and $\Upsilon >0,$ independent of $v$ or $\theta,$ 
such that 
\begin{equation}
v\left( x\right) \leq \Upsilon \left( \int_{\alpha r_{\rho }\left( x\right)
}^{\infty }\theta ^{2}\left( t\right) dt+e^{-\frac{1}{2}r_{\rho }\left(
x\right) }\theta ^{2}\left( 0\right) \right)   \label{k20}
\end{equation}%
for all $x\in M.$

Indeed, (\ref{k20}) follows in the same manner as (\ref{k17}).
Define the continuous
function 
\begin{equation*}
\varphi \left( x\right) =\left\vert \zeta \right\vert _{\infty }v^{2}
\end{equation*}%
and note that 
\begin{equation*}
0\leq \varphi \left( x\right) \leq \omega \left( r_{\rho }\left( x\right)
\right),
\end{equation*}%
where 
\begin{equation*}
\omega \left( t\right) =\left\vert \zeta \right\vert _{\infty }\theta
^{2}\left( t\right).
\end{equation*}%
By Theorem \ref{P4}, there exists a bounded solution $u\in W_{\mathrm{loc}}^{2,p}\left( M\right) $ of 
\begin{eqnarray}
\Delta u &=&-\rho \varphi  \label{k21} \\
&=&-\left\vert \zeta \right\vert _{\infty }\rho v^{2}  \notag
\end{eqnarray}%
such that 
\begin{equation*}
0<u\left( x\right) \leq C\left( \int_{\alpha r_{\rho }\left( x\right)
}^{\infty }\omega \left( t\right) dt\,+\,\mathcal{V}_{\rho }\left(
p,1\right) \omega \left( 0\right) e^{-\frac{1}{2}r_{\rho }\left( x\right)
}\right) \text{ \ on } M
\end{equation*}%
for some $0<\alpha <1.$ Using that $\omega \left( t\right) =\left\vert \zeta
\right\vert _{\infty }\theta ^{2}\left( t\right) $ and taking 
\begin{equation*}
\Upsilon :=C\left\vert \zeta \right\vert _{\infty }\max \left\{ 1,\mathcal{V}%
_{\rho }\left( p,1\right) \right\},
\end{equation*}%
we have 
\begin{equation*}
0<u\left( x\right) \leq \Upsilon \left( \int_{\alpha r_{\rho }\left(
x\right) }^{\infty }\theta ^{2}\left( t\right) dt+e^{-\frac{1}{2}r_{\rho
}\left( x\right) }\theta ^{2}\left( 0\right) \right) \text{\ on } M.
\end{equation*}%
By (\ref{k18}) and (\ref{k21}) the function $v-u\in W_{\mathrm{loc}%
}^{1,2}\left( M\right) $ is subharmonic and converges to zero at infinity.
Using the maximum principle we obtain $v\leq u$ on $M$, thus proving (\ref%
{k20}).

Fix $b>0$ small enough, depending only on $\alpha $ and $\Upsilon $ in (\ref%
{k20}), to be specified later. Note that by (\ref{k17}), there exists $%
B_{0}>0$ so that 
\begin{equation}
v\left( x\right) \leq \frac{b^{6}}{\alpha ^{2}r_{\rho }\left( x\right) +1}%
+B_{0}^{2}e^{-\alpha ^{2}r_{\rho }\left( x\right) }\text{ \ on }M.
\label{k21.1}
\end{equation}%
We prove by induction on $m\geq 2$ that 
\begin{equation}
v\left( x\right) \leq \frac{b^{2^{m}+m}}{\alpha ^{m}r_{\rho }\left( x\right)
+1}+B^{2^{m}-m}e^{-\alpha ^{m}r_{\rho }\left( x\right) }\text{ \ on }M,
\label{k22}
\end{equation}%
where $B$ is a large enough constant depending only on $\alpha,$ $\Upsilon $
and $B_{0}.$

Clearly, (\ref{k22}) holds for $m=2$ from (\ref{k21.1}). We now assume (\ref%
{k22}) holds for $m\geq 2$ and prove 
\begin{equation}
v\left( x\right) \leq \frac{b^{2^{m+1}+\left( m+1\right) }}{\alpha
^{m+1}r_{\rho }\left( x\right) +1}+B^{2^{m+1}-\left( m+1\right) }e^{-\alpha
^{m+1}r_{\rho }\left( x\right) }\text{ \ on }M.  \label{k23}
\end{equation}%
By the induction hypothesis we have $v\left( x\right) \leq \theta \left(
r_{\rho }\left( x\right) \right) ,$ where 
\begin{equation*}
\theta \left( t\right) :=\frac{b^{2^{m}+m}}{\alpha ^{m}t+1}%
+B^{2^{m}-m}e^{-\alpha ^{m}t}.
\end{equation*}%
By (\ref{k20}) we obtain that 
\begin{equation}
v\left( x\right) \leq \Upsilon \left( \int_{\alpha r_{\rho }\left( x\right)
}^{\infty }\theta ^{2}\left( t\right) dt+e^{-\frac{1}{2}r_{\rho }\left(
x\right) }\theta ^{2}\left( 0\right) \right) .  \label{k24}
\end{equation}%
Obviously,

\begin{equation}
\theta ^{2}\left( t\right) \leq \frac{2b^{2^{m+1}+2m}}{\left( \alpha
^{m}t+1\right) ^{2}}+2B^{2^{m+1}-2m}e^{-2\alpha ^{m}t}.  \label{k25}
\end{equation}%
It follows that%
\begin{eqnarray}
\int_{\alpha r_{\rho }\left( x\right) }^{\infty }\theta ^{2}\left( t\right)
dt &\leq &\frac{2}{\alpha ^{m}}\frac{b^{2^{m+1}+2m}}{\alpha ^{m+1}r_{\rho
}\left( x\right) +1}  \label{k26} \\
&&+\frac{1}{\alpha ^{m}}B^{2^{m+1}-2m}e^{-\alpha ^{m+1}r_{\rho }\left(
x\right) }.  \notag
\end{eqnarray}%
Furthermore, we have by (\ref{k25}) that 
\begin{eqnarray}
e^{-\frac{1}{2}r_{\rho }\left( x\right) }\theta ^{2}\left( 0\right) &\leq
&2\left( b^{2^{m+1}+2m}+B^{2^{m+1}-2m}\right) e^{-\frac{1}{2}r_{\rho }\left(
x\right) }  \label{k27} \\
&\leq &\frac{1}{\alpha ^{m}}B^{2^{m+1}-2m}e^{-\alpha ^{m+1}r_{\rho }\left(
x\right) }.  \notag
\end{eqnarray}%
Plugging (\ref{k26}) and (\ref{k27}) into (\ref{k24}) yields%
\begin{eqnarray}
v\left( x\right) &\leq &\frac{2\Upsilon }{\alpha ^{m}}\frac{b^{2^{m+1}+2m}}{%
\alpha ^{m+1}r_{\rho }\left( x\right) +1}+\frac{2\Upsilon }{\alpha ^{m}}%
B^{2^{m+1}-2m}e^{-\alpha ^{m+1}r_{\rho }\left( x\right) }  \label{k28} \\
&=&\left( \frac{2\Upsilon }{\alpha ^{2}}b\right) \left( \frac{b}{\alpha }%
\right) ^{m-2}\frac{b^{2^{m+1}+\left( m+1\right) }}{\alpha ^{m+1}r_{\rho
}\left( x\right) +1}  \notag \\
&&+\left( \frac{2\Upsilon }{\alpha ^{2}B}\right) \left( \frac{1}{\alpha B}%
\right) ^{m-2}B^{2^{m+1}-\left( m+1\right) }e^{-\alpha ^{m+1}r_{\rho }\left(
x\right) }.  \notag
\end{eqnarray}%
Now take $b$ sufficiently small so that $\frac{b}{\alpha }\leq 1$ and $\frac{%
2\Upsilon }{\alpha ^{2}}b\leq 1,$ and $B$ sufficiently large so that $\frac{1%
}{\alpha B}\leq 1$ and $\frac{2\Upsilon }{\alpha ^{2}B}\leq 1.$ Since $m\geq
2$, it follows by (\ref{k28}) that 
\begin{equation*}
v\left( x\right) \leq \frac{b^{2^{m+1}+\left( m+1\right) }}{\alpha
^{m+1}r_{\rho }\left( x\right) +1}+B^{2^{m+1}-\left( m+1\right) }e^{-\alpha
^{m+1}r_{\rho }\left( x\right) }.
\end{equation*}%
This proves (\ref{k23}). Hence,

\begin{equation}
v\left( x\right) \leq \frac{b^{2^{m}+m}}{\alpha ^{m}r_{\rho }\left( x\right)
+1}+B^{2^{m}-m}e^{-\alpha ^{m}r_{\rho }\left( x\right) }  \label{k29}
\end{equation}%
for all $m\geq 2.$ 

For $x\in M$ with $r_{\rho }\left( x\right)$ large, apply (\ref{k29}) by setting

\begin{equation*}
m:=\left[ \frac{\ln r_{\rho }\left( x\right) }{2\ln \left( 2\alpha
^{-1}\right) }\right],
\end{equation*}%
where $\left[ \cdot \right] $ denotes the greatest integer function. It is
not difficult to conclude that there exists constant $a>0$ such that 

\begin{equation}
v\left( x\right) \leq Ce^{-r_{\rho }^{a}\left( x\right) }\text{ \ on }M.
\label{k30}
\end{equation}

We now complete the proof of the theorem. By (\ref{k7.1}) we have that 
\begin{equation}
-\ln h\geq \frac{1}{C}e^{r_{\rho }^{a}\left( x\right) }\text{ \ on }M
\label{k31}
\end{equation}%
and satisfies 
\begin{equation*}
\Delta \left( -\ln h\right) \leq \zeta \rho.
\end{equation*}%
Consider the function 
\begin{equation*}
f\left( x\right) =\ln \left( -\ln h\right).
\end{equation*}%
Then it satisfies 
\begin{equation}
\Delta f\leq \frac{\zeta \rho }{\left( -\ln h\right) }  \label{k32}
\end{equation}%
whenever $h>0.$ Moreover, from (\ref{k31}), $f$ is bounded below by 
\begin{equation}
f\left( x\right) \geq r_{\rho }^{a}\left( x\right) -C\text{ \ on }M.
\label{k33}
\end{equation}%
Define 
\begin{equation*}
\varphi \left( x\right) =\frac{\zeta }{\left( -\ln h\right) },
\end{equation*}%
where $\varphi $ is continuously extended as $\varphi =0$ at points where $%
h=0.$ By Theorem \ref{P4} and (\ref{k31}) we can solve the Poisson equation 
\begin{eqnarray}
\Delta u &=&-\rho \varphi  \label{k34} \\
&=&-\frac{\zeta \rho }{\left( -\ln h\right) }  \notag
\end{eqnarray}%
and obtain a solution $u\in W_{\mathrm{loc}}^{2,p}\left( M\right) $ that
decays to zero at infinity.

According to (\ref{k33}), the function $f+u$ achieves its minimum at some
point $x_{0}\in M.$ Then $h\left( x_{0}\right) >0.$ So by (\ref{k32}) and (%
\ref{k34}), $f+u\in W_{\mathrm{loc}}^{1,2}\left( M\right) $ satisfies 
\begin{equation*}
\Delta \left( f+u\right) \leq 0
\end{equation*}%
in a neighborhood of $x_{0}.$ By the maximum principle, this implies that $f+u$
is constant, which is a contradiction.

Hence $h,$ as well as $\eta,$ must be identically zero on $M.$
\end{proof}

Let us point out that the hypothesis (\ref{hd}) on $\eta $ is necessary and
optimal. Indeed, consider 
\begin{equation*}
\eta \left( x\right) =e^{-\ln ^{a}\left( \left\vert x\right\vert
^{2}+e\right) }\text{ on }\mathbb{R}^{n}\text{,}
\end{equation*}%
where $0<a<1$ is fixed. It can be checked directly that 
\begin{eqnarray*}
\Delta \eta -\frac{\left\vert \nabla \eta \right\vert ^{2}}{\eta } &=&\left(
-\Delta \ln ^{a}\left( \left\vert x\right\vert ^{2}+e\right) \right) \eta \\
&\geq &-a\frac{\Delta \left\vert x\right\vert ^{2}}{\left( \left\vert
x\right\vert ^{2}+e\right) \ln ^{1-a}\left( \left\vert x\right\vert
^{2}+e\right) }\eta \\
&=&-\frac{2na}{\left( \left\vert x\right\vert ^{2}+e\right) \ln ^{1-a}\left(
\left\vert x\right\vert ^{2}+e\right) }\eta .
\end{eqnarray*}%
Now $\mathbb{R}^{n}$ satisfies weighted Poincar\'{e} inequality with
weight $\rho(x)=\frac{\left( n-2\right) ^{2}}{4}\frac{1}{\left\vert x\right\vert ^{2}%
}.$  So $\eta $ satisfies 
\begin{equation*}
\Delta \eta \geq -\zeta \rho \eta +\frac{\left\vert \nabla \eta \right\vert
^{2}}{\eta }
\end{equation*}
with
\begin{equation*}
\zeta \left( x\right) =\frac{c\left( n,a\right) }{\left( r_{\rho }\left(
x\right) +1\right) ^{1-a}}.
\end{equation*}
However, $\eta$ violates the hypothesis (\ref{hd}) as

\begin{equation*}
e^{-2c\left( n\right) \left( r_{\rho }\left( x\right) +1\right) ^{a}}\leq
\eta \left( x\right) \leq e^{-c\left( n\right) \left( \left( r_{\rho }\left(
x\right) \right) +1\right) ^{a}}.
\end{equation*}

Theorem \ref{Decay} leads to the following vanishing result for
holomorphic maps.

\begin{theorem}
\label{Vanishing}Let $\left( M^{n},g\right) $ be a complete K\"{a}hler
manifold satisfying the weighted Poincar\'{e} inequality (\ref{WP}) with
weight $\rho $ having properties (\ref{P}), (\ref{O}), (\ref{Mu1}) and 
$\rho \leq C$. Assume that the Ricci curvature has lower bound 
$\mathrm{Ric}\geq -\zeta \rho$ for some function $\zeta \left( x\right) >0$ that
converges to zero at infinity. Then any finite energy holomorphic map 
$F:M\rightarrow N,$ where $N$ is a complex Hermitian manifold of non-positive
bisectional curvature, is identically constant.
\end{theorem}

\begin{proof}
It is well known (see e.g. Theorem 1.24 in \cite{PRS}) that the differential 
$\eta =\left\vert dF\right\vert $ satisfies%
\begin{equation}
\eta \Delta \eta \geq -\zeta \rho \eta ^{2}+\left\vert \nabla \eta
\right\vert ^{2}.  \label{v1}
\end{equation}%
To be in the context of Theorem \ref{Decay}, we first show that $\eta $
decays exponentially fast in the $\rho $-distance based on the assumption
that $\int_{M}\eta ^{2}<\infty.$ Since $\zeta $ converges to zero at
infinity, by (\ref{v1}) there exists $R_{0}>0$ so that 
\begin{equation*}
\Delta \eta \geq -\frac{1}{2}\rho \eta \text{ \ on }M\backslash B_{\rho
}\left( p,R_{0}\right).
\end{equation*}%
Note that since $\rho \leq C,$ we have 
\begin{equation}
\int_{M}\rho \eta ^{2}<\infty.  \label{rho-energy}
\end{equation}%
Hence, applying Theorem 2.1 in \cite{LW1} we conclude that 
\begin{equation*}
\int_{M\backslash B_{\rho }\left( p,r\right) }\rho \eta ^{2}\leq
Ce^{-r}\int_{B_{\rho }\left( p,R_{0}\right) }\rho \eta ^{2}
\end{equation*}%
for $r\geq 2R_{0}.$

Consequently, there exists $\Lambda >0$ so that 
\begin{equation}
\int_{B_{\rho }\left( x,1\right) }\rho \eta ^{2}\leq \Lambda e^{-r_{\rho
}\left( x\right) }  \label{v2}
\end{equation}%
for all $x\in M.$ In fact, we may take $\Lambda =C\int_{M}\rho \eta ^{2}.$

We now use DeGiorgi-Nash-Moser iteration to obtain a pointwise estimate.

For $\delta $ in (\ref{O}) and $C_{0}$ in Proposition \ref{E}, fix 
\begin{equation}
r_{0}=\frac{\delta }{2C_{0}}.  \label{r0}
\end{equation}

Multiply (\ref{v1}) with $\eta ^{p-2}\phi ^{2}$, where $p\geq 2$ and $\phi
=\phi \left( r_{\rho }\left( x,\cdot \right) \right) $ is a cut-off function
with support in $B_{\rho }\left( x,r_{0}\right) $. We have

\begin{eqnarray}
\int_{M}\zeta \rho \eta ^{p}\phi ^{2} &\geq &-\int_{M}\eta ^{p-1}\phi
^{2}\Delta \eta   \label{r-1} \\
&=&\left( p-1\right) \int_{M}\left\vert \nabla \eta \right\vert ^{2}\eta
^{p-2}\phi ^{2}-2\int_{M}\left\langle \nabla \phi ,\nabla \eta \right\rangle
\eta ^{p-1}\phi   \notag \\
&\geq &\left( p-\frac{3}{2}\right) \int_{M}\left\vert \nabla \eta
\right\vert ^{2}\eta ^{p-2}\phi ^{2}-2\int_{M}\eta ^{p}\left\vert \nabla
\phi \right\vert ^{2}  \notag \\
&\geq &\frac{2p-3}{p^{2}}\int_{M}\left\vert \nabla \left( \eta ^{\frac{p}{2}%
}\phi \right) \right\vert ^{2}-3\int_{M}\eta ^{p}\left\vert \nabla \phi
\right\vert ^{2}.  \notag
\end{eqnarray}%
Using the Sobolev inequality from Lemma \ref{L} for $B_{\rho }\left(x,r_{0}\right) $ 
we get that

\begin{eqnarray}
&&\int_{B_{\rho }\left( x,r_{0}\right) }\left\vert \nabla \left( \eta ^{%
\frac{p}{2}}\phi \right) \right\vert ^{2}  \label{r0.1} \\
&\geq &\frac{1}{C}\rho \left( x\right) \mathrm{V}\left( B_{\rho }\left(
x,r_{0}\right) \right) ^{\frac{2}{n}}\left( \int_{B_{\rho }\left(
x,r_{0}\right) }\eta ^{\frac{np}{n-2}}\phi ^{\frac{2n}{n-2}}\right) ^{\frac{%
n-2}{n}}  \notag \\
&&-C\rho \left( x\right) \int_{B_{\rho }\left( x,r_{0}\right) }\eta ^{p}\phi
^{2}.  \notag
\end{eqnarray}%
Plugging (\ref{r0.1}) into (\ref{r-1}) and noting that

\begin{equation*}
\left\vert \nabla \phi \right\vert ^{2}=\left( \phi ^{\prime }\right)
^{2}\left\vert \nabla r_{\rho }\left( x,\cdot \right) \right\vert ^{2}=\rho
\left( x\right) \left( \phi ^{\prime }\right) ^{2}
\end{equation*}%
and 
\begin{equation}
\sup_{B_{\rho }\left( x,r_{0}\right) }\rho \leq C\inf_{B_{\rho }\left(
x,r_{0}\right) }\rho   \label{v3}
\end{equation}
by Proposition \ref{E}, we obtain 

\begin{equation*}
\mathrm{V}\left( B_{\rho }\left( x,r_{0}\right) \right) ^{\frac{2}{n}}\left(
\int_{B_{\rho }\left( x,r_{0}\right) }\eta ^{\frac{np}{n-2}}\phi ^{\frac{2n}{%
n-2}}\right) ^{\frac{n-2}{n}}\leq C\int_{B_{\rho }\left( x,r_{0}\right)
}\eta ^{p}\left( \phi ^{2}+\left( \phi ^{\prime }\right) ^{2}\right).
\end{equation*}

The standard Moser iteration then gives

\begin{equation*}
\eta ^{2}\left( x\right) \leq \frac{C}{\mathrm{V}\left( B_{\rho }\left(
x,r_{0}\right) \right) }\int_{B_{\rho }\left( x,r_{0}\right) }\eta ^{2}.
\end{equation*}%
Together with (\ref{v3}), this yields 
\begin{equation}
\eta ^{2}\left( x\right) \leq \frac{1}{\mathcal{V}_{\rho }\left(
x,r_{0}\right) }\int_{B_{\rho }\left( x,r_{0}\right) }\rho \eta ^{2}.
\label{v6}
\end{equation}%
According to Theorem \ref{V} and (\ref{Mu1}) we have 
\begin{equation*}
\mathcal{V}_{\rho }\left( x,r_{0}\right) \geq \frac{1}{C}v_{0}>0
\end{equation*}%
for all $x\in M.$ Then (\ref{v6}) and (\ref{v2}) imply that 
\begin{equation}
\eta \left( x\right) \leq \Lambda e^{-\frac{1}{2}r_{\rho }\left( x\right) }
\label{v6.1}
\end{equation}%
for all $x\in M,$ where $\Lambda $ is a constant depending on the total
energy of $\eta $ on $M.$

Applying Theorem \ref{Decay}, we conclude $\eta =0$ and $F$ is a constant
map.
\end{proof}

We point out that in \cite{LY1} Li and Yau proved a vanishing theorem for holomorphic 
maps $F:M\rightarrow N,$ where $M$ is assumed to be non-parabolic and its Ricci 
curvature is bounded from below by $\mathrm{Ric}\geq -\bar{\rho}$ with $\bar{\rho}$ 
being an integrable function. An alternative proof of this result using the Poisson equation is given
as Theorem 8.6 in \cite{PRS}.

As a consequence of Theorem \ref{Vanishing} we obtain the following structural result.

\begin{corollary}
\label{KE}Let $\left( M^{n},g\right) $ be a complete manifold satisfying the
weighted Poincar\'{e} inequality (\ref{WP}) with weight $\rho $ having
properties (\ref{P}), (\ref{O}), (\ref{Mu1}) and $\rho
\leq C.$ Assume that the Ricci curvature is bounded by $\mathrm{Ric}\geq -\zeta
\rho$ for some function $\zeta \left( x\right) >0$ that converges to zero
at infinity. Then $M$ has only one end.
\end{corollary}

\begin{proof}
Let us assume by contradiction that $M$ has at least two ends. We denote
by $E$ a nonparabolic end and let $F=M\backslash E.$ Note
that $E$ exists because $M$ is nonparabolic. We claim that $F$ is
nonparabolic as well. Indeed, if $F$ were parabolic, then by \cite{LW1},

\begin{equation*}
\int_{\left( M\backslash B_{\rho }\left( p,R\right) \right) \cap F}\rho
\left( y\right) dy\leq C\,e^{-2R}
\end{equation*}%
for all $R.$ This obviously contradicts with (\ref{Mu1}).
Hence, both $E$ and $F$ are nonparabolic ends. By Li-Tam \cite{LT},
there exists a harmonic function $w$ on $M$ with the following
properties.
\begin{eqnarray}
\int_{M}\left\vert \nabla w\right\vert ^{2} &<&\infty  \label{u} \\
\limsup_{F}w &=&1  \notag \\
\liminf_{E}w &=&0.  \notag
\end{eqnarray}%
Such $w$ is necessarily pluriharmonic according to \cite{L}. Therefore,
Theorem \ref{Vanishing} is applicable to $w$ and $w$ must be constant.
This shows that $M$ must be connected at infinity. 
\end{proof}

\section{The special case of constant weight \label{Co}}

In this section we specialize to the case when $\rho =\lambda _{1}\left(\Delta \right)$ 
and present an alternative approach from \cite{MSW} to Theorem \ref{G1}. The argument
relies on the heat kernel estimates and is more streamlined. Since it avoids the level set
consideration, such an approach may be applicable to more general setting. 
In the following, $C$ denotes a constant depending only on $n,$ $K$ and $\lambda _{1}\left( \Delta \right).$ 
Denote by $H\left( x,y,t\right) $ the minimal heat kernel of $M.$ 

Let us restate Theorem \ref{G1} below.

\begin{theorem}
\label{Heat}Let $\left( M^{n},\,g\right) $ be a Riemannian manifold with
positive spectrum $\lambda _{1}\left( \Delta \right) >0$ and with Ricci
curvature $\mathrm{Ric}\geq -K$ for some constant $K\geq 0.$ Then there exists 
$C>0$ such that for any $p, x\in M$ and any $r>0,$ 

\begin{equation*}
\int_{B\left( p,r\right) }G\left( x,y\right) dy\leq C\left( r+1\right).
\end{equation*}
\end{theorem}

\begin{proof}
As noted in the proof of Theorem \ref{G3}, it suffices to prove the result
for $x\in B\left( p,r\right).$ 

It is well known (see e.g. Chapter 10 in \cite{Gr1}) that 

\begin{equation*}
e^{\lambda _{1}\left( \Delta \right) t}H\left( x,x,t\right) \text{ \ is
nonincreasing in }t>0\text{.}
\end{equation*}
Therefore,

\begin{equation}
H\left( x,x,t\right) \leq e^{-\lambda _{1}\left( \Delta \right) \left(
t-1\right) }H\left( x,x,1\right)  \label{l4}
\end{equation}%
for all $t\geq 1.$ Using the
semi-group property and the Cauchy-Schwarz inequality, we get

\begin{eqnarray*}
H\left( x,y,2t\right) &=&\int_{M}H\left( x,z,t\right) H\left( y,z,t\right) dz\\
&\leq &\left( \int_{M}H\left( x,z,t\right) ^{2}dz\right) ^{\frac{1}{2}%
}\left( \int_{M}H\left( y,z,t\right) ^{2}dz\right) ^{\frac{1}{2}} \\
&=&H\left( x,x,2t\right) ^{\frac{1}{2}}H\left( y,y,2t\right) ^{\frac{1}{2}}.
\end{eqnarray*}%
Together with (\ref{l4}), this proves that 
\begin{equation}
H\left( x,y,t\right) \leq e^{-\lambda _{1}\left( \Delta \right) \left(
t-1\right) }H\left( x,x,1\right) ^{\frac{1}{2}}H\left( y,y,1\right) ^{\frac{1%
}{2}}  \label{l5}
\end{equation}%
for all $x, y\in M$ and all $t\geq 1.$ 

By Li-Yau \cite{LY} we have for all $x\in M$ 
\begin{equation}
H\left( x,x,1\right) \leq \frac{C}{\mathrm{V}\left( x,1\right) },  \label{l6}
\end{equation}%
where $\mathrm{V}\left( x,1\right) =\mathrm{Vol}\left( B\left( x,1\right)\right).$ 
However, by the Bishop-Gromov volume comparison theorem,
for any $x\in B\left( p,r\right),$

\begin{eqnarray*}
\frac{\mathrm{V}\left( p,r\right) }{\mathrm{V}\left( x,1\right) } &\leq &%
\frac{\mathrm{V}\left( x,2r\right) }{\mathrm{V}\left( x,1\right) } \\
&\leq &e^{Cr}.
\end{eqnarray*}%
Hence, if both $x, y\in B\left( p,r\right),$ then we get from (\ref{l6}) that 
\begin{equation*}
H\left( x,x,1\right) ^{\frac{1}{2}}H\left( y,y,1\right) ^{\frac{1}{2}}\leq
e^{Cr}\,\mathrm{V}\left( p,r\right) ^{-1}.
\end{equation*}%
Plugging this into (\ref{l5}) we conclude that 
\begin{equation*}
H\left( x,y,t\right) \leq C\,e^{-\lambda _{1}\left( \Delta \right) t+Cr}%
\mathrm{V}\left( p,r\right) ^{-1}
\end{equation*}%
for all $x,y\in B\left( p,r\right)$ and $t\geq 1.$
This immediately implies that for some $C_{1}>0,$ 

\begin{equation}
\int_{B\left( p,r\right) }H\left( x,y,t\right) dy\leq C_{1}e^{-\lambda
_{1}\left( \Delta \right) t+C_{1}r}  \label{l7}
\end{equation}%
for any $x\in B\left( p,r\right) $ and $t\geq 1.$ In particular, for $t\geq \Lambda$
with 
\begin{equation}
\Lambda =\max \left\{ 1,\frac{2C_{1}r}{\lambda _{1}\left( \Delta \right) }%
\right\},  \label{l8}
\end{equation}
one has

\begin{equation*}
\int_{B\left( p,r\right) }H\left( x,y,t\right) dy\leq Ce^{-\frac{1}{2}%
\lambda _{1}\left( \Delta \right) t}
\end{equation*}%
for all $x\in B\left( p,r\right).$  We integrate this
inequality from $t=\Lambda$ to $t=\infty $ and use Fubini's theorem to
conclude that 
\begin{equation}
\int_{B\left( p,r\right) }\left( \int_{\Lambda }^{\infty }H\left(
x,y,t\right) dt\right) dy\leq C  \label{l9}
\end{equation}%
for any $x\in B\left( p,r\right).$ On the other hand, it is well know that
the minimal heat kernel satisfies 

\begin{equation*}
\int_{M}H\left( x,y,t\right) dy\leq 1
\end{equation*}%
for all $x\in M.$ It implies that

\begin{eqnarray*}
\int_{B\left( p,r\right) }\left( \int_{0}^{\Lambda }H\left( x,y,t\right)
dt\right) dy &=&\int_{0}^{\Lambda }\left( \int_{B\left( p,r\right) }H\left(
x,y,t\right) dy\right) dt \\
&\leq &\Lambda.
\end{eqnarray*}%
In view of the choice of $\Lambda$ from (\ref{l8}) we conclude that 
\begin{equation}
\int_{B\left( p,r\right) }\left( \int_{0}^{\Lambda }H\left( x,y,t\right)
dt\right) dy\leq C\left( r+1\right)   \label{l10}
\end{equation}%
for all $x\in B\left( p,r\right).$

Combining (\ref{l9}) and (\ref{l10}), we obtain that

\begin{equation*}
\int_{B\left( p,r\right) }\left( \int_{0}^{\infty }H\left( x,y,t\right)
dt\right) dy\leq C\left( r+1\right) 
\end{equation*}%
for all $x\in B\left( p,r\right).$ Since
\begin{equation*}
G\left( x,y\right) =\int_{0}^{\infty }H\left( x,y,t\right) dt,
\end{equation*}%
this shows
\begin{equation*}
\int_{B\left( p,r\right) }G\left( x,y\right) dy\leq C\left( r+1\right) 
\end{equation*}%
for all $x\in B\left( p,r\right).$ The theorem is proved.
\end{proof}

\end{document}